\documentclass[12pt]{article}
\usepackage{a4wide, amstext, amsmath, amssymb, graphicx, array, epsfig, psfrag}
\usepackage{amstext, amsmath, amssymb, graphicx}

\newcommand{\bbR}{{\mathbb{R}}}

\newcommand{\bfI}{\boldsymbol I}

\newcommand{\bfV}{\boldsymbol V}

\newcommand{\mcK}{\mathcal{K}}

\newcommand{\mcF}{\mathcal{F}}
\newcommand{\mcN}{\mathcal{N}}
\newcommand{\mcA}{\mathcal{A}}
\newcommand{\mcU}{\mathcal{U}}

\newcommand{\tn}{|\mspace{-1mu}|\mspace{-1mu}|}

\newcommand{\Gammah}{{\Gamma_h}}
\newcommand{\nablag}{{\nabla_\Gamma}}

\newcommand{\curvten}{\mathcal{H}_\Gamma}
\numberwithin{equation}{section}
\newtheorem{lem}{Lemma}[section]

\newtheorem{thm}{Theorem}[section]

\newenvironment{proof}{\noindent \newline {\bf Proof.}}
{\hfill \mbox{\fbox{} } \newline}
\date{}
\begin{document}
\title{\bf Cut Finite Element Methods for Coupled Bulk-Surface Problems}
\author{Erik Burman
\footnote{Department of Mathematics, University College London, London, UK-WC1E  6BT, United Kingdom} 
\footnote{Supported by EPSRC, UK, Grant Nr. EP/J002313/1.}
\mbox{ }
 Peter Hansbo
\footnote{Department of Mechanical Engineering, J\"onk\"oping University,
SE-55111 J\"onk\"oping, Sweden.} 
\footnote{Supported by the Swedish Foundation for Strategic Research Grant Nr.\ AM13-0029 and the Swedish Research Council Grant Nr.\ 2011-4992.}
\mbox{ }
Mats G.\ Larson
\footnote{Department of Mathematics and Mathematical Statistics, Ume{\aa} University, SE-90187 Ume{\aa}, Sweden} 
\footnote{Supported by the Swedish Foundation for Strategic Research Grant Nr.\ AM13-0029 and the Swedish Research Council Grant Nr.\ 2013-4708.}
\mbox{ }
Sara Zahedi
\footnote{Department of Mathematics, KTH Royal Institute of Technology,
SE-100 44 Stockholm, Sweden}
}
\numberwithin{equation}{section} \maketitle
 \begin{abstract}
We develop a cut finite element method for a second order elliptic 
coupled bulk-surface model problem. We prove a priori estimates for the 
energy and $L^2$ norms of the error. Using stabilization terms we show that 
the resulting algebraic system of equations has a similar condition
number as a standard fitted finite element method. Finally, we
present a numerical example illustrating the accuracy and the
robustness of our approach. 
\end{abstract}

\section{Introduction}
Problems involving phenomena that take place both on surfaces (or
interfaces) and in bulk domains occur in a variety of applications in
fluid dynamics and biological applications. An example is
given by the modeling of soluble surfactants.  Surfactants are
important because of their ability to reduce the surface
tension. Examples of applications where the effects of surfactants are important in the modelling include detergents, oil recovery, and the treatment of lung diseases.
A soluble surfactant is dissolved in the bulk fluid but also exists in adsorbed form on the interface.  A computational challenge is then to properly account for the exchange between these two surfactant forms.
 The coupling between the dissolved form in the bulk and the adsorbed
 form on the interface involves computations of the gradient of the
 bulk surfactant concentration on a moving interface that may undergo
 topological changes, see e.g.\cite{BooSie10}. In this context computational 
 methods that allow the interface to be arbitrarily located with respect 
 to a fixed background mesh are of great interest.

We consider a basic model problem of this nature that involves two coupled elliptic 
problems one in the bulk and one on the boundary of the bulk domain. The 
coupling term is defined in such a way that the overall bilinear form 
in the corresponding weak statement is coercive. A finite element
method was proposed and analyzed for a  similar model problem
in~\cite{EllRan12}. See also \cite{DziEll13}, and the references 
therein for background on finite element methods for partial differential 
equations on surfaces. In~\cite{EllRan12} a polyhedral approximation of the
bulk domain was used and its piecewise polynomial boundary faces
served as
approximation of the surface. In this contribution we develop a method
that is unfitted, that is, the surface is allowed to cut through a
fixed background mesh in an arbitrary way.  Such a finite element
method was proposed in~\cite{ORG09} for the Laplace--Beltrami
operator. 
A general framework for this type of computational methods using
finite element methods on cut meshes, co called CutFEM methods was
recently discussed in \cite{BuClHaLaMa14}. The CutFEM approach is 
convenient since the same finite element space defined on a background 
grid can be used for solving both the partial differential equation 
in the bulk region and on the surface. However, a drawback of this 
type of methods is that the stiffness matrix may become arbitrarily 
ill conditioned depending on the position of the surface in the background 
mesh. In the case of the Laplace--Beltrami operator this ill conditioning 
has been addressed in~\cite{OlsReu10} and~\cite{BurHanLar13}. For results
on the stability of the bulk equation on cut meshes see \cite{BurHan12, 
HanLar13, MasLar14}.

We use continuous piecewise linear elements defined on the 
background mesh to solve both the problem in the bulk domain and the problem 
on the surface. To stabilize the method we add gradient jump penalty terms as in~\cite{BurHan12,BurHanLar13} that ensure that the resulting algebraic system of equations has optimal 
condition number. We also consider the approximation of the domain and prove 
a priori error estimates in both the $H^1$-- and $L^2$--norms, taking both the approximation of the domain and of the solution into account.

The remainder of the paper is outlined as follows: In Section 2 we 
introduce the model problem and state the weak form, in Section 3 we 
introduce a discrete approximation of the domain, in Section 4 we prove 
a priori estimates for the energy and $L^2$ norm of the error, in Section 
5 we prove an estimate of the condition number, and finally 
in Section 6 we present a numerical example.

\section{The Continuous Coupled Bulk-Surface Problem}

\subsection{Strong Form}

Let $\Omega$ be a domain in $\bbR^3$ with smooth boundary $\Gamma$ 
and exterior unit normal $n$. We consider the following problem: find 
$u_B: \Omega \rightarrow \bbR$ and $u_S: \Gamma \rightarrow {\bbR}$ 
such that
\begin{alignat}{3}\label{eq:LBa}
-\nabla\cdot (k_B \nabla u_B) &= f_B \quad &\text{in $\Omega$}  
\\ \label{eq:LBb}
-n\cdot k_B \nabla u_B &=   {b_B} u_B - {b_S} u_S  \quad &  \text{on $\Gamma$}
\\ \label{eq:LBc}
-\nabla_\Gamma \cdot(k_S \nabla_\Gamma u_S) &= f_S - n \cdot k_B \nabla u_B 
\quad & \text{on $\Gamma$}
\end{alignat}
Here $\nabla$ is the $\bbR^3$ gradient and $\nabla_\Gamma$ is the 
tangent gradient associated with $\Gamma$ defined by
\begin{equation}
\nabla_\Gamma = P_\Gamma \nabla 
\end{equation}
with $P_\Gamma = P_\Gamma(x)$ the projection of $\bbR^3$ onto the tangent 
plane of $\Gamma$ at $x\in\Gamma$, defined by
\begin{equation}
P_\Gamma = I -n \otimes n
\end{equation}
Further, 
${b_B}$, ${b_S}$, $k_B$, and $k_S$ are positive constants, and
$f_B:\Omega \rightarrow \bbR$ and $f_S:\Gamma \rightarrow \bbR$ are given functions. 
{As mentioned above, this problem serves as a basic model for the concentration of 
surfactants interacting with a bulk concentration; it also models other processes, e.g., proton transport via a membrane surface \cite{GeMeSt02}.}

\subsection{Weak Form}
Multiplying (\ref{eq:LBa}) by $v_B \in H^1(\Omega)$, integrating by parts, 
and using the boundary condition (\ref{eq:LBb}), we obtain
\begin{align}
(f_B,v_B)_\Omega &= (k_B \nabla u_B, \nabla v_B )_\Omega 
- (n \cdot k_B \nabla u_B, v_B )_{\Gamma}
\\
&= (k_B \nabla u_B, \nabla v_B )_\Omega 
+ ({b_B} u_B - {b_S} u_S, v_B)_{\Gamma}
\end{align}
and thus we have the weak statement
\begin{equation}
(k_B \nabla u_B, \nabla v_B )_\Omega 
+ ({b_B} u_B - {b_S} u_S, v_B )_{\Gamma}  = (f_B,v_B)_\Omega
\quad \forall v_B \in H^1(\Omega)
\end{equation}
Next multiplying (\ref{eq:LBc}) by $v_S \in H^1(\Gamma)$, integrating by parts, 
and again using (\ref{eq:LBb}) we obtain 
\begin{align}
(k_S \nabla_\Gamma u_S, \nabla_\Gamma v_S )_\Gamma 
&= (f_S - n \cdot k_S \nabla u_B, v_S )_\Gamma 
\\
&= ( f_S + ({b_B} u_B - {b_S} u_S), v_S)_{\Gamma}
\end{align}
and thus
\begin{equation}
(k_S \nabla u_S, \nabla v_S )_\Gamma - ({b_B} u_B - {b_S} u_S, v_S)_{\Gamma}
= (f_S,v_S)_\Gamma \quad \forall v_S \in H^1(\Gamma)
\end{equation}
We note that the solution to this system of equations is uniquely determined 
up to a pair of constant functions $(c_B,c_S)$ such that 
${b_B} c_B - {b_S} c_S=0$. To obtain a unique solution we here choose to enforce $\int_\Gamma u_S=0$.

Introducing the function spaces
\begin{equation}
V_B = H^1(\Omega) , \quad V_S = H^1(\Gamma)/\langle 1_\Gamma\rangle,\quad W = V_B \times V_S
\end{equation}
and choosing the test functions ${b_B} v_B$ and ${b_S} v_S$ we 
get the variational problem: find $u = (u_B,u_S) \in W$ such that
\begin{equation}
a(u,v) = l(v) \quad \forall v \in W
\end{equation}
Here
\begin{equation}
a(u,v) = a_B( u_B , v_B)  + a_S(u_S,v_S) + a_{BS}(u,v)
\end{equation}
with
\begin{equation}\label{eq:contformsa}
\begin{cases}
a_B( u_B , v_B)={b_B} ( k_B\nabla u_B , \nabla v_B )_\Omega 
\\
 a_S(u_S,v_S) = {b_S} (k_S \nabla_S  u_S , \nabla_S  v_S )_\Gamma
 \\
 a_{BS}(u,v) = ({b_B} u_B -{b_S} u_S,{b_B} v_B  - {b_S} v_S )_\Gamma
 = (b \cdot u,b \cdot v )_\Gamma
\end{cases}
\end{equation}
where we also introduced the notation $b=(b_B,-b_S)$ and
\begin{equation}\label{eq:contformsl}
l(v) = l_B(v_B) + l_S(v_S) = b_B(f_B,v_B)_\Omega + b_S(f_S,v_S)_\Gamma
\end{equation}
Introducing the energy norm 
\begin{equation}
\tn u \tn^2 = a(u,u)
\end{equation}
we directly obtain coercivity and continuity of the bilinear form 
$a(\cdot,\cdot)$ and continuity of $l(\cdot)$. Using Lax-Milgram there is 
a unique solution in $W$. If $\Gamma$ is $C^3$ we additionally have the elliptic 
regularity estimate
\begin{equation}
\| u_B \|_{H^2(\Omega)} + \| u_S \|_{H^2(\Gamma)} \lesssim \| f_B \|_{L^2(\Omega)} 
+ \| f_S \|_{L^2(\Gamma)} 
\end{equation}
see \cite{EllRan12} for details. Here and below $\lesssim$ denotes 
less or equal up to a constant, $\| \cdot \|_{H^s(\omega)}$ denotes the 
standard Sobolev norm $H^s(\omega)$ norm on the set $\omega$, and 
$\|\cdot \|_{L^p(\omega)}$ denotes 
the $L^p(\omega)$ norm. 

\section{The Finite Element Method}

\subsection{Approximation of the Domain}
\begin{figure}
\begin{center} 
\includegraphics[width=0.7\textwidth]{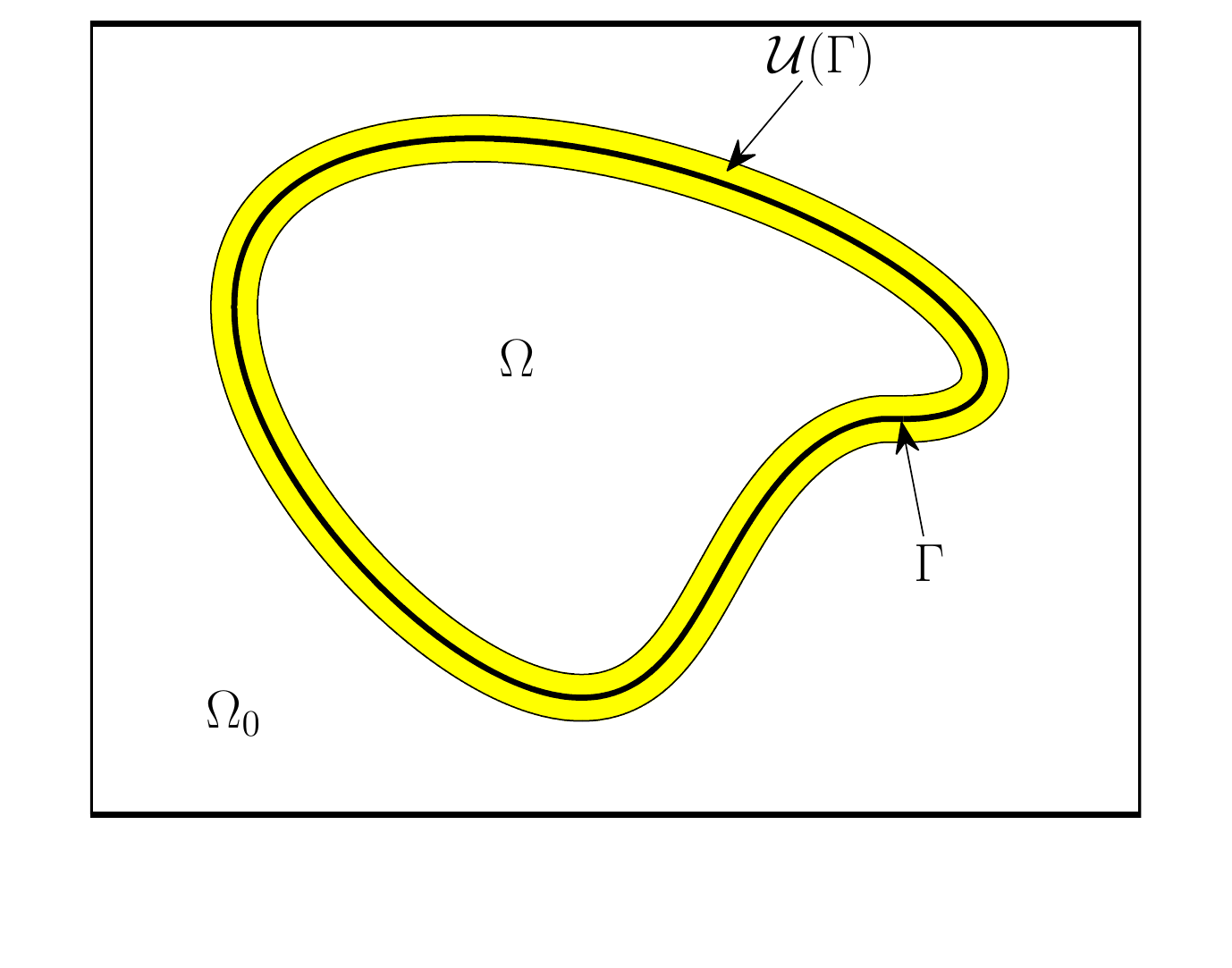}
\caption{Illustration of the domain $\Omega$, $\Omega_0$, $\mcU(\Gamma)$, and  $\Gamma$. The domain $\mcU(\Gamma)$ is the yellow region where for each $x \in \mcU(\Gamma)$ there is a unique closest point on $\Gamma$.  \label{fig:illust}}
\end{center}
\end{figure}
Let $p: \bbR^3 \ni x \mapsto \text{argmin}_{y \in \Gamma} 
|y - x| \in \Gamma$ denote the closest point mapping. Then there 
is an open neighborhood $\mcU(\Gamma)$ of $\Gamma$ such that for 
each $x \in \mcU(\Gamma)$ there is a uniquely determined $p(x)\in \Gamma$. 
We let $\rho$ be the signed distance function, $\rho(x) = | p(x) -  x |$ 
in $\bbR^3 \setminus \Omega$ and $\rho(x) = - |p(x) - x|$ in $\Omega$. We define 
the extension of any function define on $\Gamma$ to $\mcU(\Gamma)$ as follows 
\begin{equation}\label{eq:extension}
v^e = v \circ p
\end{equation}

Let $\Omega_0$ be a domain in $\bbR^3$ that contains $\Omega \cup 
\mcU(\Gamma)$ and let $\mcK_{0,h}$ be a quasiuniform partition of 
$\Omega_0$ into shape regular tetrahedra with mesh parameter $h$. See Fig.~\ref{fig:illust} for an illustration of the different domains. We consider a continuous piecewise linear approximation $\Gamma_h$ of $\Gamma$ such that $\Gamma_h \cap K$ is a subset of a hyperplane in $\bbR^3$ for each $K \in \mcK_{0,h}$. 

We assume that $\Gamma_h \subset \mcU(\Gamma)$ and that the following 
approximation assumptions hold:
\begin{equation}\label{eq:geomassumptiona}
\| \rho \|_{L^\infty(\Gamma_h)} \lesssim h^2
\end{equation}
and
\begin{equation}\label{eq:geomassumptionb}
\| n^e - n_h \|_{L^\infty(\Gamma_h)} \lesssim h
\end{equation}
where $n_h$ denotes the piecewise constant exterior unit normal to 
$\Gamma_h$. Finally, we define $\Omega_h$ as the domain enclosed by 
$\Gamma_h$. These assumptions are consistent with the piecewise linear 
nature of the discrete surface.

\subsection{Finite Element Spaces}

We define the following sets of elements
\begin{equation}
\mcK_{B,h} = \{K \in \mcK_{h,0} \,:\, K \cap \Omega_h \neq \emptyset \},
\qquad 
\mcK_{S,h} = \{K \in \mcK_{h,0} \,:\, K \cap \Gamma_h \neq \emptyset \}
\end{equation}
and the corresponding sets
\begin{equation}\label{eq:mcnmcs}
\mcN_{B,h} = \bigcup_{K \in \mcK_{B,h}} K,
\qquad
\mcN_{S,h} = \bigcup_{K \in \mcK_{S,h}} K  
\end{equation}
We let $V_{0,h}$ be the space of piecewise linear continuous functions defined 
on $\mcK_{0,h}$. Next let 
\begin{equation}
V_{B,h} = V_{0,h}|_{\mcN_{B,h}} ,
\qquad V_{S,h}=V_{0,h}|_{\mcN_{S,h}}/ \langle 1_{\Gamma_h}\rangle,
\qquad W_h = V_{B,h} \times V_{S,h}
\end{equation}
be the spaces of continuous piecewise linear polynomials defined on 
$\mcN_{B,h}$ and $\mcN_{S,h}$, respectively, where we also enforced 
$\int_{\Gamma_h} v_S = 0$ for $v \in V_{S,h}$.

\subsection{The Finite Element Method}

The finite element method takes the form: find $u_h = (u_{B,h},u_{S,h}) \in W_h$ 
such that
\begin{equation}\label{eq:fem}
A_h(u_h , v) = l_h(v) \quad \forall v \in W_h
\end{equation}
Here the bilinear form is defined by
\begin{equation}
A_h(v,w) = a_h(v,w) + j_h(v,w)
\end{equation}
with 
\begin{equation}
a_h(v,w) = a_{B,h}(v_B,w_B) + a_{S,h}(v_S,w_S) + a_{BS,h}(v,w)
\end{equation}
and 
\begin{equation}\label{eq:discformsa}
\begin{cases}
a_{B,h}( u_B , v_B)={b_B} ( k_B\nabla u_B , \nabla v_B )_{\Omega_h} 
\\
 a_{S,h}(u_S,v_S) = {b_S} (k_S \nabla_S  u_S , \nabla_S  v_S )_{\Gamma_h}
 \\
 a_{BS,h}(u,v) = ({b_B} u_B -{b_S} u_S,{b_B} v_B  - {b_S} v_S )_{\Gamma_h}
 =(b\cdot u,b \cdot v )_{\Gamma_h}
\end{cases}
\end{equation}
where $\nabla_{\Gamma_h} = P_h \nabla$ and $P_h = I - n_h \otimes n_h$. Next $j_h(v,w)$ is a stabilizing term of the form
\begin{equation}\label{eq:discformsj}
j_h(v,w )= \tau_B h^3 j_B(v_B,w_B) + \tau_S j_S (v_S,w_S)
\end{equation}
where $\tau_B, \tau_S$ are positive parameters and, letting $[x]\vert_F$
denote the jump of $x$ over the face $F$,
\begin{align}
j_B(v_B,w_B) &= \sum_{F \in \mcF_{B,h}} ([n_F\cdot \nabla v_B],[n_F\cdot \nabla w_B])_F
\\
j_S(v_S,w_S) &= \sum_{F \in \mcF_{S,h}} ([n_F \cdot \nabla v_S],[n_F\cdot \nabla w_S])_F
\end{align}
with $\mcF_{S,h}$ the set of internal faces (i.e. faces with two 
neighbors) in $\mcK_{S,h}$ and $\mcF_{B,h}$ denotes the set of faces that 
are internal in $\mcK_{B,h}$ and belong to an element in $\mcK_{S,h}$. 
Finally, the right hand side is defined by
\begin{equation}\label{eq:discformsl}
l_h(v) = l_{B,h}(v_B) + l_{S,h}(v_S) =  b_B(f_{B,h},v_B)_{\Omega_h} 
+ b_S(f_{S,h},v_S)_\Gammah
\end{equation}
with $f_{B,h}$ and $f_{S,h}$ discrete approximations of $f_B$ and 
$f_S$ that will be specified more precisely below.

The purpose of the stabilization terms is to ensure that the resulting 
algebraic system of equations is well conditioned.


\section{A Priori Error Estimates}
\paragraph{Outline of the proof.} To prove a priori error estimates we 
first construct a bijective mapping $F_h$ that maps the exact domain 
to the approximate domain. The mapping is used to lift the 
discrete solution onto the exact domain where the error is evaluated. 
The construction of the mapping is based on a representation of the 
discrete boundary $\Gamma_h$ as 
a normal function over the exact boundary $\Gamma$ together with an extension 
to a small tubular  $\delta$ neighborhood of the boundary. In the remainder 
of the domain $F_h$ is the identity mapping. Next a Strang 
type lemma relates the error in the computed solution to an interpolation 
error and quadrature errors emanating from the approximation of the domain. 
Using the assumptions on the approximation properties of the discrete surface 
we derive bounds on 
the quadrature errors. The surface quadrature errors are $O(h^2)$ while 
the bulk quadrature error is $O(h)$ in the $\delta$ neighborhood and zero elsewhere. To establish an optimal order energy norm error estimate only 
first order estimates of the quadrature errors are needed but for $L^2$ 
error estimates second order estimates are necessary. To achieve a 
second order estimate of the quadrature error we utilize the fact that 
$\delta$ can be chosen in the form $\delta = C h$ with a sufficiently large 
$C$.
\subsection{Mapping the Exact Domain to the Approximate Domain}

\paragraph{The Mapping $\boldsymbol{F_h}$:} For $\delta>0$ let 
$\mcU_\delta(\Gamma)$ be the open tubular $\delta$ neighborhood 
\begin{equation}
\mcU_\delta(\Gamma )= \{ x \in \bbR^3: |\rho(x)|< \delta \} 
\end{equation}
For $0<\delta\leq \delta_0$, where $\delta_0$ is a constant, that only 
depend on the domain, chosen such that 
$\mcU_{\delta_0}(\Gamma) \subset \mcU(\Gamma)$, 
the mapping
\begin{equation}
\mcU_\delta (\Gamma) \ni x \mapsto (p(x),\rho(x))\in \Gamma \times (-\delta,\delta) 
\end{equation}
is a bijection with inverse
\begin{equation}
\Gamma \times (-\delta,\delta) \ni (x,z) \mapsto x + z n(x) \in \mcU_\delta (\Gamma)
\end{equation}
We next note that there is a function $\gamma_h: \Gamma \rightarrow \bbR$ 
such that 
\begin{equation}\label{eq:gammahdef}
q_h: \;\Gamma \ni x \mapsto x + n(x) \gamma_h(x) \in \Gamma_h
\end{equation}
is a bijection. {Since for $x \in \Gamma_h$ there holds $p(x) = x -
n^e(x) \rho(x)$ we may deduce that $q_h(x)$ is the inverse mapping to
$p(x):\Gamma_h \mapsto \Gamma$.} Using the assumptions on the approximation properties 
(\ref{eq:geomassumptiona}) and (\ref{eq:geomassumptionb}) we obtain 
the following estimates {(see Appendix)}
\begin{equation}\label{eq:gammahest}
\|\gamma_h \|_{L^\infty(\Gamma)}\lesssim h^2, 
\qquad \|\nablag \gamma_h \|_{L^\infty(\Gamma)} \lesssim h
\end{equation}
Assuming that $h$ is sufficiently small so that $\Gamma_h \subset \mcU_{\delta/3}(\Gamma)$ we may define the mapping
\begin{equation}\label{eq:defFh}
F_h: \Omega_0 \ni x \mapsto x + \chi(\rho(x)) n^e(x) \gamma_h^e (x) \in \Omega_0  
\end{equation}
where $\chi:(-\delta,\delta) \rightarrow [0,1]$ is a smooth cut off function 
that equals $1$ on $(-\delta/3,\delta/3)$ and $0$ on 
$(-\delta,\delta)\setminus(-2\delta/3,2\delta/3)$ and the derivative $D \chi$ satisfies the estimate 
\begin{equation}\label{eq:Dchiest}
\|D \chi \|_{L^\infty(-\delta,\delta)} \lesssim \delta^{-1}
\end{equation}
We note that by construction $F_h:\Omega_0 \rightarrow \Omega_0$ is a bijection 
such that  
\begin{equation}
F_h(\Omega) = \Omega_h, \qquad F_h(\Gamma) = \Gamma_h
\end{equation}
and 
\begin{equation}
F_h= I \quad \text{in $\Omega_0 \setminus \mcU_{\delta}(\Gamma)$}
\end{equation}

\paragraph{The Derivative $\boldsymbol{DF_h}$:}
The derivative 
$DF_h(x) \in \mathcal{L}(\bbR^3,\bbR^3)$ of $F_h$ at $x\in \Omega_0$ 
is given by
\begin{align}
DF_h(x) &=I + \Big( \chi(\rho(x)) n^e(x) \Big) D (\gamma_h^e(x)) 
\\ \nonumber
&\qquad + \Big(D( \chi(\rho(x)) n^e(x) ) \Big) \gamma_h^e(x)
\\
&=I + \Big( \chi (\rho(x)) n^e(x) \Big) (D \gamma_h)^e(x) Dp(x) 
\\ \nonumber
&\qquad + \Big((D \chi)(\rho(x)) D \rho(x) n^e(x) \Big) \gamma_h^e(x)
\\ \nonumber
&\qquad
+ \Big( \chi(\rho(x)) (Dn)^e(x) Dp(x) \Big) \gamma_h^e(x)
\end{align}
Next we note that 
\begin{equation}
D \rho = n^e, \qquad Dn=\curvten, \qquad  Dp = P_\Gamma^e - \rho \curvten
\end{equation}
where we used the identity $p(x) = x - \rho(x) D\rho(x)  =x - \rho(x) n^e(x)$ and introduced the curvature tensor $\curvten(x) = \nabla \otimes \nabla \rho(x), x \in \Gamma$. Note that it holds $\|\curvten \|_{L^\infty(\mcU_\delta(\Gamma)})\lesssim 
1$ for $\delta$ small enough. Thus we have
\begin{align}\label{eq:DFh}
DF_h(x) 
&=I + \chi (\rho(x)) n^e(x) \otimes (\nabla_\Gamma \gamma_h)^e(x) (P_\Gamma^e(x) - \rho(x) \curvten(x)) 
\\ \nonumber
&\qquad + \gamma_h^e(x)(D \chi)(\rho(x)) n^e(x) \otimes n^e(x) 
\\ \nonumber
&\qquad
+ \chi(\rho(x)) \gamma_h^e(x) \curvten^e(x) (P_\Gamma^e(x) - \rho(x) \curvten(x)) 
\end{align}
%
%
%
%
%
%
On the surface $\Gamma$ we have the simplified expression
\begin{equation}\label{eq:DFhSurface}
DF_h(x) = I + n(x) \otimes \nabla_\Gamma \gamma_h(x) +  \gamma_h(x) 
\curvten(x) 
\end{equation}
since $\chi = 1$ and $D\chi=0$ in a neighborhood of $\Gamma$ and 
$\rho(x) = 0$ for $x\in \Gamma$.
We note that $DF_h(x)$ maps the tangent space $T_x(\Gamma)$ into the
piecewise defined tangent space $T_{F_h(x)}(\Gamma_h)$. In other 
words we have the identity 
\begin{equation}\label{eq:Fhtan}
DF_h P_\Gamma = (P_\Gammah\circ F_h) DF_h P_\Gamma
\end{equation}
and the mapping 
\begin{equation}\label{eq:DFhGamma}
DF_{h,\Gamma}(x):T_x(\Gamma) \ni y \mapsto 
(P_\Gammah \circ F_h) DF_h P_\Gamma y \in T_{F_h(x)}(\Gamma_h)
\end{equation}
is invertible. Observing that by \eqref{eq:gammahest}, $DF_h = I +
O(h)$, for small enough $h$ we have the bounds
\begin{equation}\label{eq:DFhbounds}
\|DF_{h} \|_{L^\infty(\Omega_0,\mathcal{L}(\bbR^3,\bbR^3))} \lesssim 1,
\qquad
\|DF^{-1}_{h} \|_{L^\infty(\Omega_0,\mathcal{L}(\bbR^3,\bbR^3))} \lesssim 1
\end{equation}
and
\begin{equation}\label{eq:DFhboundsgamma}
\|DF_{h,\Gamma} \|_{L^\infty(\Gamma,\mathcal{L}(T_x(\Gamma),T_{F_h(x)}(\Gamma_h))} \lesssim 1,
\qquad
\|DF^{-1}_{h,\Gamma} \|_{L^\infty(\Gamma,\mathcal{L}(T_{F_h(x)}(\Gamma_h),T_x(\Gamma))} \lesssim 1
\end{equation}
Below we simplify the notation as follows 
$\|DF_{h} \|_{L^\infty(\Omega_0)} = \|DF_{h} \|_{L^\infty(\Omega_0,\mathcal{L}(\bbR^3,\bbR^3))}$ for the mappings $DF_h$ and $DF_{h,\Gamma}$ and their inverses.

\paragraph{The Jacobian Determinants $\boldsymbol{JF_{h}}$ 
and $\boldsymbol{JF_{h,\Gamma}}$:} We have the following relations 
between the measures on the exact and approximate surface and domain
\begin{equation}
d \Omega_h = JF_{h} d \Omega, 
\qquad  
d \Gamma_h = JF_{h,\Gamma} d \Gamma
\end{equation}
where the Jacobian determinants are defined by
\begin{align}
JF_{h}(x) &= | \text{det}(DF_h(x))|
\\
JF_{h,\Gamma}(x) &= |DF_{h,\Gamma}(x)\xi_1 \times DF_{h,\Gamma}(x)\xi_2|
\end{align}
and $\{\xi_1,\xi_2\}$ is an orthonormal basis in $T_x(\Gamma)$.
We note that $JF_{h} = 1$ on $\Omega_0 \setminus \mcU_\delta (\Gamma)$ 
and recall that $DF_h = I + O(h)$. Thus we 
have the following estimates in the bulk
\begin{equation}\label{eq:Jacdetbulk}
\|JF_{h}\|_{L^\infty(\Omega_0)} \lesssim 1,\qquad
\|JF^{-1}_{h}\|_{L^\infty(\Omega_0)} \lesssim  1,\qquad
\|1 - JF_{h}\|_{L^\infty(\mcU_{\delta}(\Gamma))} \lesssim h
\end{equation}
since the determinant is a third order polynomial of the elements in 
$DF_h$. On the surface we note that 
\begin{equation}
DF_{h,\Gamma}(x)\xi = \xi + n (\xi \cdot  \nabla_\Gamma \gamma_h ) 
+ \gamma_h \curvten \cdot \xi \qquad \forall \xi \in T_x(\Gamma)
\end{equation}
where the last term is $O(h^2)$.
The Jacobian determinant $JF_{h,\Gamma}$ is the norm of the cross product
\begin{align}
|DF_{h,\Gamma}(x)\xi_1 \times DF_{h,\Gamma}(x)\xi_2| &= 
|(\xi_1 + n (\xi_1 \cdot  \nabla_\Gamma \gamma_h )) 
\times (\xi_2 + n (\xi_2 \cdot  \nabla_\Gamma \gamma_h )) | + O(h^2)
\nonumber \\
&=|n - \xi_1 (\xi_1 \cdot \nabla_\Gamma \gamma_h )
- \xi_2 (\xi_2 \cdot \nabla_\Gamma \gamma_h ) | + O(h^2)
\nonumber \\
&=\Big( 1 + (\xi_1 \cdot \nabla_\Gamma \gamma_h)^2 
+ (\xi_2 \cdot \nabla_\Gamma \gamma_h)^2  \Big)^{1/2} + O(h^2)
\nonumber \\
&=1 + O(h^2)
\end{align} 
where we used the identities $\xi_1 \times \xi_2 = n$, 
$n\times \xi_2 = -\xi_1$, $\xi_1 \times n = - \xi_2$, 
$n\times n = 0$, the fact that $\{\xi_1,\xi_2,n\}$ is a positively oriented orthonormal basis in $\bbR^3$ to compute the norm, and finally the estimate $(1+\delta)^{1/2} \leq 1+\delta/2$, $\delta>0$ in the last step. We thus have the following 
estimates for the surface Jacobian
\begin{equation}\label{eq:Jacdetsurf}
\|JF_{h,\Gamma}\|_{L^\infty(\Gamma)} \lesssim 1,\qquad
\|JF^{-1}_{h,\Gamma}\|_{L^\infty(\Gamma)} \lesssim  1,\qquad
\|1 - JF_{h,\Gamma}\|_{L^\infty(\Gamma)} \lesssim h^2
\end{equation}

\subsection{Lifting to the Exact Domain}

We define the lifting or pullback  of $v^L$ with respect to $F_h$ 
of a function $v$ defined on $\Omega_0$ as follows
\begin{equation}
v^L := v \circ F_h
\end{equation}
We note in particular that any function defined on 
$\Omega_h$ and $\Gamma_h$ may be lifted to a function 
on $\Omega$ and $\Gamma$.
%
Using the chain rule
\begin{equation}
D v^L = D(v \circ F_h) = (D v \circ F_h ) DF_h = (Dv)^L DF_h
\end{equation}
and thus we obtain the identities
\begin{equation}
\nabla v^L = DF_h^T (\nabla v \circ F_h )= DF_h^T (\nabla v)^L
\end{equation}
\begin{align}
\nablag v^L 
&= P_\Gamma \nabla v^L
= P_\Gamma DF_h^T (\nabla v )^L 
\nonumber \\ 
&\qquad 
= P_\Gamma DF_h^T P_\Gammah^L (\nabla v )^L
= (P_\Gamma DF_h^T P_\Gammah^L) (\nabla_\Gammah v )^L 
= DF_{h,\Gamma}^T  (\nabla_\Gammah v )^L 
\end{align}
where $DF_{h,\Gamma}$ was defined in (\ref{eq:DFhGamma}). 
Summarizing, we have the relations 
\begin{equation}\label{eq:LiftDer}
\nabla v^L = DF_h^T (\nabla v)^L,\qquad (\nabla v)^L = DF_h^{-T} \nabla v^L
\end{equation}
and
\begin{equation}\label{eq:LiftDerGamma}
\nabla_\Gamma v^L = DF_{h,\Gamma}^T (\nabla_\Gammah v )^L,
\qquad
(\nabla_\Gammah v )^L = DF_{h,\Gamma}^{-T}\nabla_\Gamma v^L
\end{equation}
Using the bounds (\ref{eq:DFhbounds}) and (\ref{eq:DFhboundsgamma}) we 
conclude that the following equivalences hold
\begin{equation}\label{eq:equivalencebulk}
\| \nabla v^L \|_{L^2(\Omega)} \lesssim \| (\nabla v)^L \|_{L^2(\Omega)} 
\lesssim \| \nabla v^L \|_{L^2(\Omega)}
\end{equation}
and
\begin{equation}\label{eq:equivalencesurf}
\| \nabla_\Gamma v^L \|_{L^2(\Gamma)} \lesssim \| (\nabla_\Gammah v)^L \|_{L^2(\Gamma)} 
\lesssim \| \nabla_\Gamma v^L \|_{L^2(\Gamma)}
\end{equation}

\subsection{Interpolation}

Let $E_B:H^2(\Omega) \rightarrow H^2(\Omega_0)$ be an extension operator 
such that 
\begin{equation}\label{eq:stabEB}
\|E_B v\|_{H^2(\Omega_0)} \lesssim \| v \|_{H^2(\Omega)} 
\end{equation}
and $E_S:H^2(\Gamma) \rightarrow H^2(\mcU(\Gamma))$ be the extension 
operator such that $E_S v = v \circ p$. Then we have the estimate
\begin{equation}\label{eq:stabES}
\|E_S v\|_{H^2(\mcU_\delta(\Gamma))} \lesssim \delta^{1/2}\| v \|_{H^2(\Gamma)} 
\end{equation}
for any $\delta>0$ such that $\mcU_\delta(\Gamma) \subset \mcU(\Gamma)$. We 
finally define the extension operator
\begin{equation}
E:H^2(\Omega) \times H^2(\Gamma)\ni (u_B,u_S) \mapsto (E_B u_B, E_S u_S) 
\in H^2(\Omega_0) \times H^2(\mcU(\Gamma))
\end{equation}
When suitable we simplify the notation and write $u=Eu$. We let $\pi_{SZ,h}: L^2(\Omega_0) \rightarrow V_{0,h}$ denote the standard 
Scott-Zhang interpolation operator and recall the interpolation error estimate
\begin{equation}\label{interpolstandard}
\| v - \pi_{SZ,h} v \|_{H^m(K)} \leq C h^{2-m} \| v \|_{H^2(\mcN(K))}, \quad m = 1,2, \quad K \in \mcK_{0,h}
\end{equation}
where $\mcN(K)\subset \Omega_h$ is the union of the neighboring elements of 
$K$. We then define the interpolant 
\begin{equation}
\pi_h u = (\pi_{B,h} u_B,  \pi_{S,h} u_S)
\end{equation}
where
\begin{equation}
\pi_{B,h} u_B = (\pi_{SZ,h} E_B u_B)|_{\mcN_{B,h}} \in  V_{B,h}
\end{equation}
and
\begin{equation}
 \pi_{S,h} u_S= (\pi_{SZ,h} E_S u_S)|_{\mcN_{S,h}} \in V_{S,h}
\end{equation}
We use the notation 
\begin{equation}\label{eq:interpollift}
\pi_h^L u = (\pi_h u)^L = (\pi_h u) \circ F_h
\end{equation}
for the pullback of $\pi_h u$ to $\Omega$ by $F_h$.
With these definitions we 
have the following lemma:
\begin{lem}\label{lem:approx} The following estimate holds
\begin{equation}\label{eq:interpolest}
\tn u - \pi_h^L u \tn  \lesssim h\| u \|_{H^2(\Omega) \times H^2(\Gamma)}
\end{equation}
\end{lem}
\begin{proof} Using a trace inequality we obtain
\begin{align}
\tn u - \pi_h^L u \tn^2 &= b_Bk_B\| \nabla (u_B - \pi_{B,h}^L u_B)\|^2_{L^2(\Omega)}
 + b_Sk_S \| \nabla (u_S - \pi_{S,h}^L u_S)\|^2_{L^2(\Gamma)}
\nonumber \\ 
&\qquad
+ \| b_B (u_B - \pi_{B,h}^L u_B) - b_S (u_S - \pi_{S,h}^L u_S) \|^2_{L^2(\Gamma)}  
\nonumber \\
&\lesssim \| u_B - \pi_{B,h}^L u_B\|^2_{H^1(\Omega)} 
+ \| u_S - \pi_{S,h}^L u_S\|^2_{H^1(\Gamma)}
\nonumber\\
&=I+II
\end{align}
\paragraph{Term $\bfI$.}
The first term may be estimated as follows
\begin{align}
I&=\| u_B - \pi_{B,h}^L u_B\|_{H^1(\Omega)} =\| u_B -(\pi_{SZ,h} E_B u_B|_{\Omega_h} )^ L\|_{H^1(\Omega)} 
\nonumber\\
&\leq \label{eq:interpollemmaa}
\| u_B - (E_B u_B|_{\Omega_h})^L \|_{H^1(\Omega)} 
+ \|((I-\pi_{SZ,h}) E_B u_B|_{\Omega_h})^L\|_{H^1(\Omega)}
\nonumber\\
&\lesssim h \| u_B \|_{H^2(\Omega)}
\end{align}
{Here we used the Sobolev Taylor's formula, see \cite{BreSco08},
to estimate the first 
term: consider first a function $v\in H^2(\Omega_0)$; then we have }
\begin{equation}
\| v - v \circ F_h \|_{L^2(\Omega_0)} 
\lesssim  
\|I-F_h\|_{L^\infty(\Omega_0)} \| \nabla v \|_{L^2(\Omega_0)}
\lesssim
h^2 \| v \|_{H^1(\Omega_0)}
\end{equation} 
and for the derivative
\begin{align}
&\| \nabla (v - v \circ F_h) \|_{L^2(\Omega_0)} 
\nonumber\\
&\qquad=\| \nabla v - DF_h^T (\nabla v \circ F_h) \|_{L^2(\Omega_0)}
\nonumber\\
&\qquad\leq \| \nabla v - (\nabla v \circ F_h)\|_{L^2(\Omega_0)} 
+ \| (DF_h^T -I) (\nabla v \circ F_h) \|_{L^2(\Omega_0)}
\nonumber\\
&\qquad\lesssim  
\|I-F_h\|_{L^\infty(\Omega_0)} \| \nabla v \|_{H^1(\Omega_0)}
+
\|(DF_h^T - I)\|_{L^\infty(\Omega_0)} \| \nabla v \|_{L^2(\Omega_0)}
\nonumber\\
&\qquad\lesssim h^2 \| v \|_{H^2(\Omega_0)} + h \| \nabla v \|_{L^2(\Omega_0)}
\nonumber\\
&\qquad \lesssim h \| v \|_{H^2(\Omega_0)}
\end{align}
Now we may apply these inequalities with $v=E_B u_B$ and finally use 
the stability (\ref{eq:stabEB}) of the extension operator $E_B$. 

The 
second term in (\ref{eq:interpollemmaa}) is estimated by mapping to the 
discrete domain using the interpolation estimate (\ref{interpolstandard}) 
and then using the stability estimate (\ref{eq:stabEB}).

\paragraph{Term $\bfI\bfI$.} Changing domain of integration 
from $\Gamma$ to $\Gamma_h$ and then using an element--wise trace inequality
we obtain
\begin{align}
\| \nabla_\Gamma(u_S - \pi_{S,h}^L u_S) \|_{L^2(\Gamma)}^2 
&=\|DF_{h,\Gamma}^{T} \nabla_{\Gamma_h}(u_S^e- \pi_{S,h} u_S) |JF_{h,\Gamma}|^{-1/2} \|^2_{L^2(\Gamma_h)}
\nonumber\\
&\lesssim \sum_{K \in \mcK_{S,h}} h^{-1} \| u_S^e- \pi_{S,h} u_S \|_{H^1(K)} 
+ h \| u_S^e- \pi_{S,h} u_S \|_{H^2(K)}^2
\nonumber\\
&\lesssim \sum_{K \in \mcK_{S,h}} h \|u_S^e\|^2_{H^2(\mcN(K))}
\nonumber\\
&\lesssim h^2 \| u_S \|^2_{H^2(\Gamma)}
\end{align}
Here we used the interpolation estimate (\ref{interpolstandard}) followed 
by the stability estimate (\ref{eq:stabES}) for the extension operator 
with $\delta \sim h$, which is possible since there is $\delta \lesssim h$ such that $\mcK_{S,h}\subset\mcU_\delta(\Gamma)$. 
\end{proof}

We also need the face norm
\begin{align}
\tn v \tn_\mcF^2 &= h^3 j_B(v_{B,h},v_{B,h}) + j_S(v_{S,h},v_{S,h}) 
\\
&=\sum_{F \in \mcF_{B,h}} h^3 \|[n_F\cdot \nabla v_B]\|_{L^2(F)}^2
+\sum_{F \in \mcF_{S,h}} \|[n_F \cdot \nabla v_S\|_{L^2(F)}^2
\end{align}
for which we have the following interpolation error estimate.
\begin{lem}\label{lem:approxface} The following estimate holds
\begin{equation}\label{eq:interpolestF}
\tn u - \pi_h u \tn_\mcF  \lesssim h\| u \|_{H^2(\Omega) \times H^2(\Gamma)}
\end{equation}
\end{lem}
\begin{proof}
This estimate follows directly by using an element wise trace inequality, 
followed by the interpolation estimate (\ref{interpolstandard}), and finally
the stability estimates (\ref{eq:stabEB}) and (\ref{eq:stabES}) for the extension 
operators.
\end{proof}
\subsection{Strang's Lemma}

\begin{lem}\label{lem:strang} The following estimate holds
\begin{align}
\Big( \tn u-u_h^L \tn^2 + \tn u - u_h \tn_{\mcF}^2 \Big)^{1/2}
&\lesssim \Big( \tn u - \pi^L_h u  \tn^2 +  \tn u - \pi_h u \tn_{\mcF}^2 \Big)^{1/2}
\\ \nonumber 
&\qquad + \sup_{v \in W_h} \frac{a(u_h^L,v^L) - a_h(u_h,v)}{\tn v^L \tn}\nonumber
\\ \nonumber
&\qquad + \sup_{v \in W_h} \frac{l( v^L) - l_h(v)}{\tn v^L \tn}
\end{align} 
\end{lem}
\begin{proof} Adding and subtracting an interpolant $\pi_h^L u$, defined 
by (\ref{eq:interpollift}), and using the triangle inequality we obtain
\begin{align}
 \Big(\tn u - u_h^L \tn^2 + \tn u - u_h \tn_{\mcF}^2 \Big)^{1/2}
 &\leq \Big( \tn u - \pi_h^L u \tn_h^2 + \tn u - \pi_h u \tn_{\mcF}^2 \Big)^{1/2} 
 \\ \nonumber
&\qquad  + \Big( \tn \pi_h^L u  - u_h^L \tn_h +\tn \pi_h u - u_h \tn_{\mcF}^2\Big)^{1/2}
\end{align}
To estimate the second term we start from the coercivity
\begin{equation}\label{eq:infsup}
\Big(\tn \pi_h^L u  - u_h^L \tn^2 + \tn \pi_h u - u_h \tn_{\mcF}^2 \Big)^{1/2} 
\leq \sup_{v \in W_h \setminus \{0\}} 
\frac{a(\pi_h^L u  - u_h^L,v^L)+ j_h(\pi_h u  - u_h,v)}{\Big(\tn v^L \tn^2 + \tn v \tn_{\mcF}^2\Big)^{1/2}}
\end{equation}
Adding and subtracting the exact solution, and using Galerkin orthogonality the numerator may be written in the following form
\begin{align}
&a( \pi_h^L u  - u_h^L,v^L) + j_h(\pi_h u  - u_h,v)
\nonumber \\ \nonumber
&\qquad=  
a( \pi_h^L u - u,v^L) + a(u - u_h^L,v^L)+ j_h(\pi_h u  - u_h,v)
\\ \nonumber
&\qquad=a( \pi_h^L u  - u,v^L) + l(v^L) - a(u_h^L,v^L)+ j_h(\pi_h u  - u_h,v)
\\ \nonumber 
&\qquad=a( \pi_h^L u - u,v^L) + l(v^L) -l_h(v) 
\\ \nonumber
&\qquad \qquad+ a_h(u_h,v)  + j_h(u_h,v)- a(u_h^L,v^L) + j_h(\pi_h u  - u_h,v)
\\ \nonumber
&\qquad =a( \pi_h^L u - u,v^L) + j_h(\pi_h u  - u,v)  
\\ 
&\qquad \qquad+ \Big( a_h(u_h,v) - a(u_h^L,v^L) \Big) 
+ \Big( l(v^L) -l_h(v) \Big)
\end{align}
Using (\ref{eq:infsup}) and estimating the first term using the
Cauchy-Schwarz inequality the lemma follows directly.
\end{proof}

\subsection{Estimate of the Quadrature Errors}
\begin{lem}\label{lem:quada} 
If $h \lesssim  \delta \leq \delta_0$ and $h$ is small enough. Then 
it holds
\begin{align}
|a(v^L,w^L) - a_h(v,w)|
\lesssim {} & h^2 \| \nabla_\Gamma v_S^L \|_{L^2(\Gamma)} 
\| \nabla_\Gamma w_S^L \|_{L^2(\Gamma)} 
\\ \nonumber
& + h^2 \|b \cdot v^L \|_{L^2(\Gamma)} \|b \cdot w^L\|_{L^2(\Gamma)} 
\\ \nonumber
&
+
h \| \nabla v_B^L \|_{L^2(\mcU_{\delta}(\Gamma)\cap \Omega)} 
\| \nabla w_B^L \|_{L^2(\mcU_{\delta}(\Gamma)\cap \Omega)} \quad \forall v,w\in W_h
\end{align} 
\end{lem}

\begin{proof} Using the definition of the bilinear forms we have
\begin{align}
a(v^L,w^L)- a_h(v,w)= {} &\underbrace{a_B(v_B^L,w_B^L) - a_{B,h}(v_B,w_B)}_{I}
\nonumber \\ 
& + \underbrace{a_S(v_S^L,w_S^L)-a_{S,h}(v_S,w_S) }_{II} + \underbrace{a_{BS}(v,w)-a_{BS,h}(v,w)  }_{III} 
\nonumber \\
={} &I + II + III
\end{align}
We now proceed with estimates of the three terms.

\paragraph{Term $\bfI$.} Starting from the definition of the forms 
(\ref{eq:contformsa}) and (\ref{eq:discformsa}), changing domain 
of integration to $\Omega$, and using (\ref{eq:LiftDer}), we obtain
the following identity
\begin{align}\nonumber
&({b_B} k_B)^{-1} (a_B(v_B^L,w_B^L) - a_{B,h} (v_B, w_B)) 
\nonumber \\
&\qquad = 
( DF_h^T (\nabla v_B)^L, DF_h^T (\nabla w_B)^L )_\Omega
- (\nabla v_B, \nabla w_B )_{\Omega_h}
\nonumber  \\
&\qquad= ( DF_h^T (\nabla v_B)^L, DF_h^T (\nabla w_B)^L )_\Omega
-((\nabla v_B)^L, (\nabla w_B)^L JF_{h} )_{\Omega}
\nonumber  \\
&\qquad= ( ( DF_h DF_h^T- JF_{h} I) (\nabla v_B)^L,  (\nabla w_B)^L )_\Omega
\nonumber  \\
&\qquad=(\mcA_{h,\Omega} (\nabla v_B)^L,(\nabla w_B)^L )_\Omega
\end{align}

In order to estimate 
$\mcA_{h,\Omega}=DF_h DF_h^T- JF_{h} I$ we note that 
$\mcA_{h,\Omega} = 0$ in 
$\Omega_0 \setminus \mcU_\delta(\Gamma)$ and in $\mcU_\delta(\Gamma)$ 
we have the identity
\begin{align}
\mcA_{h,\Omega} &= DF_h DF_h^T- JF_{h} I 
\nonumber  \\
&= (DF_h - I ) ( DF_h - I )^T + (DF_h + DF_h^T) - I - JF_{h} I
 \nonumber  \\
&= (DF_h - I ) ( DF_h - I )^T + (DF_h-I) + (DF_h-I)^T + (1 - JF_{h}) I
\end{align}
and therefore we have the estimate
\begin{align}
\|\mcA_{h,\Omega}\|_{L^\infty(\mcU_\delta(\Gamma)\cap \Omega)} 
&\lesssim 
\|DF_h - I \|^2_{L^\infty(\mcU_\delta(\Gamma)\cap \Omega)} 
\\ \nonumber
&\qquad + \|DF_h - I \|_{L^\infty(\mcU_\delta(\Gamma)\cap \Omega)} 
+ \| 1 - JF_{h} \|_{L^\infty(\mcU_\delta(\Gamma)\cap \Omega)}
\end{align}
This estimate holds for any $0<\delta\leq \delta_0$ and $h$ 
such that 
\begin{equation}\label{eq:defFhdelta}
\Gamma_h \subset \mcU_{\delta/3}(\Gamma)
\end{equation}
Recall that (\ref{eq:defFhdelta}) is required in the 
definition (\ref{eq:defFh}) of the mapping $F_h$. Now using 
the assumption that there is a constant $C_1>0$ 
such that $C_1 h < \delta \leq \delta_0$, there is a constant 
$h_0>0$, independent of $\delta$, such that (\ref{eq:defFhdelta}) 
holds for $0<h\leq h_0$, since we have the estimate 
$\|\gamma_h \|_{L^\infty(\Gamma)} \leq  C_2 h^2 \leq (C_2 h_0) h 
< C_1 h/3 < \delta/3$, where we may choose $h_0$ such that 
$C_2 h_0 < C_1/3$. 

Proceeding with the estimate of $\|DF_h - I \|_{L^\infty(\mcU_\delta(\Gamma)\cap \Omega)}$ 
for $C_1 h < \delta \leq \delta_0$ and $0<h\leq h_0$ we start 
from the identity (\ref{eq:DFh}) and then using the estimates 
$\|\chi\|_{L^\infty(-\delta,\delta)}=1, 0<\delta \leq \delta_0$ 
and $\|P_\Gamma^e - \rho \curvten\|_{L^\infty(\mcU_{\delta_0}(\Gamma))} 
\lesssim 1$ we obtain
\begin{align}\nonumber
 \|DF_h - I \|_{L^\infty(\mcU_\delta(\Gamma)\cap \Omega)}
&\lesssim \|\nabla_\Gamma \gamma_h\|_{L^\infty(\Gamma)}
\\ \nonumber
&\qquad + \|\gamma_h\|_{L^\infty(\Gamma)}\|D \chi\|_{L^\infty(-\delta,\delta)} 
+ \|\gamma_h\|_{L^\infty(\Gamma)} 
\\ \nonumber
&\lesssim h + h^2 \delta^{-1} + h^2
\\ \label{eq:estDFh}
&\lesssim h 
\end{align}
where we used (\ref{eq:gammahest}) and (\ref{eq:Dchiest}) and 
$C_1 h < \delta$. This estimate 
holds for all $\delta$ and $h$ such that 
$C_1 h < \delta \leq \delta_0$ and $0<h \leq h_0$. 
Combining (\ref{eq:estDFh}) with the estimate for the Jacobian determinant 
(\ref{eq:Jacdetbulk}) we obtain the estimate 
\begin{equation}\label{eq:estAhOmega}
\| \mcA_{h,\Omega} \|_{L^\infty(\mcU_\delta(\Gamma)\cap \Omega)}
\lesssim h
\end{equation}
and we also recall that
\begin{equation}
\mcA_{h,\Omega} = 0 \quad \text{in $\Omega \setminus \mcU_\delta (\Gamma)$}
\end{equation}

Using the bound (\ref{eq:estAhOmega}) for $\mcA_{h,\Omega}$ we obtain 
the estimate
\begin{align}
|a_B(v^L,w^L) - a_{B,h} (v, w)|
&\lesssim h \| (\nabla v)^L \|_{L^2(\mcU_\delta(\Gamma))}
\| (\nabla w)^L \|_{{L^2(\mcU_\delta(\Gamma))}}
\nonumber \\
&\lesssim h \| \nabla v^L \|_{L^2(\mcU_\delta(\Gamma))}
\| \nabla w^L \|_{L^2(\mcU_\delta(\Gamma))}
\end{align}
At last we used the estimate 
\begin{align}
\| (\nabla v)^L \|_{L^2(\mcU_\delta(\Gamma))} 
= {} &
\| DF_h^{-T} (\nabla v^L) \|_{L^2(\mcU_\delta(\Gamma))}
\\ \nonumber
={} &
\| DF_h^{-T} \|_{L^\infty(\mcU_\delta(\Gamma))} \|\nabla v^L \|_{L^2(\mcU_\delta(\Gamma))}
\lesssim 
\| \nabla v^L \|_{L^2(\mcU_\delta(\Gamma))}
\end{align}
where we employed (\ref{eq:DFhbounds}).

\paragraph{Term $\bfI\bfI$.} Proceeding in the same way and using 
(\ref{eq:LiftDerGamma}) we obtain
\begin{align}\nonumber
&({b_S} k_S)^{-1} \Big(a_S(v_S^L,w_S^L) - a_{S,h} (v_S, w_S)\Big) 
\nonumber \\
&\qquad = 
( \nabla_\Gamma v_S^L, \nabla_\Gamma w_S^L )_\Gamma
- (\nabla_\Gammah v_S, \nabla_\Gammah w_S )_{\Gamma_h}
\nonumber \\
&\qquad = 
( DF_{h,\Gamma}^T (\nabla_\Gammah v_S)^L, DF_{h,\Gamma}^T (\nabla_\Gammah w_S)^L )_\Gamma
 - ((\nabla_\Gammah v_S)^L, (\nabla_\Gammah w_S)^L JF_{h,\Gamma})_{\Gamma}
\nonumber \\
&\qquad = 
( ( DF_{h,\Gamma} DF_{h,\Gamma}^T - P_\Gammah^L JF_{h,\Gamma} ) 
(\nabla_\Gammah v_S)^L, 
(\nabla_\Gammah w_S)^L )_\Gamma
\nonumber\\ 
&\qquad = (\mcA_{\Gamma,h} (\nabla_\Gammah v_S)^L,(\nabla_\Gammah w_S)^L )_\Gamma
\end{align}
where we introduced
\begin{align}\label{eq:AGammah}
\mcA_{\Gamma,h} = DF_{h,\Gamma} DF_{h,\Gamma}^T - P_\Gammah^L JF_{h,\Gamma}
\end{align}
Using the definition (\ref{eq:DFhGamma}) of $DF_{h,\Gamma}$ and the expression (\ref{eq:DFhSurface}) for $DF_h$ we have the identity
\begin{align} 
DF_{h,\Gamma} &=P_\Gammah^{{L}} DF_h P_\Gamma
\nonumber \\
 &= P_\Gammah^{{L}}(I + n \otimes \nabla_\Gamma \gamma_h 
+ \gamma_h \curvten) P_\Gamma
\nonumber \\
&= P_\Gammah^{{L}} P_\Gamma + (P_\Gammah^{{L}} n) \otimes \nabla_\Gamma \gamma_h 
+ \gamma_h P_\Gammah^{{L}} \curvten P_\Gamma
\end{align}
Here the second term can be estimated as follows
\begin{equation}\label{eq:normal_outerprod}
\|(P_\Gammah^{{L}} n) \otimes \nabla_\Gamma \gamma_h\|_{L^{\infty}(\Gamma)}
\lesssim
\|P_\Gammah^{{L}} n\|_{L^{\infty}(\Gamma)} 
\| \nabla_\Gamma \gamma_h\|_{L^{\infty}(\Gamma)}
\lesssim h^2
\end{equation}
where we used the estimate
\begin{equation}\label{est:PGammahn}
\|P_\Gammah^L n\|_{L^{\infty}(\Gamma)} = \|P_\Gammah^L (n -
n^L_h)\|_{L^{\infty}(\Gamma)} \lesssim \|n \circ p -
n_h\|_{L^\infty(\Gamma_h)} \lesssim h
\end{equation}
For the third term we have the estimate
\begin{equation}
\|\gamma_h P_\Gammah^{{L}} \curvten P_\Gamma\|_{L^\infty(\Gamma)} 
\lesssim 
\| \gamma_h \|_{L^\infty(\Gamma)} \|P_\Gammah^{{L}}\|_{L^\infty(\Gamma)}
\|\curvten\|_{L^\infty(\Gamma)}
\|P_\Gamma\|_{L^\infty(\Gamma)}
\lesssim h^2
\end{equation}
Thus we conclude that 
\begin{equation}
DF_{h,\Gamma}=
P_\Gammah^{{L}} P_\Gamma + O(h^2)
\end{equation}
Inserting this identity into the expression (\ref{eq:AGammah}) for $\mcA_{\Gamma,h}$ and using the identity 
\begin{equation}
P_\Gammah^L JF_{h,\Gamma} = P_\Gammah^L + P_\Gammah^L (JF_{h,\Gamma}-1) 
= P_\Gammah^L + O(h^2)
\end{equation}
where we used (\ref{eq:Jacdetsurf}), we obtain
\begin{align}
\mcA_{\Gamma,h} 
&= P_\Gammah^{{L}} P_\Gamma P_\Gammah^{{L}} - P_\Gammah^{{L}} + O(h^2)
\end{align}
Now the following identity holds
\begin{equation}
P_\Gammah^{{L}} P_\Gamma P_\Gammah^{{L}} - P_\Gammah^{{L}} 
= P_\Gammah^{{L}} ( P_\Gamma - P_\Gammah^{{L}}) ( P_\Gamma - P_\Gammah^{{L}}) P_\Gammah^{{L}}
\end{equation}
which leads to the estimate
\begin{equation}\label{eq:projgam_gamh}
\|P_\Gammah^{{L}} P_\Gamma P_\Gammah^{{L}} - P_\Gammah^{{L}} \|_{L^\infty(\Gamma)} 
\leq \|P_\Gammah^{{L}}\|^2_{L^\infty(\Gamma)} \| P_\Gamma - P_\Gammah^{{L}}\|^2_{L^\infty(\Gamma)} 
\lesssim h^2
\end{equation}
where we used the bound
\begin{align}\nonumber
\| P_\Gamma - P_\Gammah^L\|_{L^\infty(\Gamma)} = {} & \|  n \otimes n -
n_h^L \otimes n_h^L\|_{L^\infty(\Gamma)} 
\\ \nonumber
 \lesssim {} & \|(n - n_h^L)
\otimes n\|_{L^\infty(\Gamma)}  + \|n_h^L \otimes (n -
n_h^L)\|_{L^\infty(\Gamma)} 
\\
\nonumber
 \lesssim {} & \|n^e -n_h\|_{L^\infty(\Gamma_h)} 
\\
\nonumber
\lesssim {}& h
\end{align}
Thus we finally arrive at 
\begin{equation}
\|\mcA_{\Gamma,h} \|_{L^\infty(\Gamma)} \lesssim h^2
\end{equation}
and therefore we have the estimate
\begin{align}
|a_S(v^L,w^L) - a_{S,h} (v, w)|
&\lesssim h^2
\|(\nabla_\Gammah v)^L\|_{L^2(\Gamma)} 
\|(\nabla_\Gammah w)^L \|_{L^2(\Gamma)}
\nonumber \\
&\lesssim h^2
\|\nablag v^L\|_{L^2(\Gamma)} \|\nablag w^L \|_{L^2(\Gamma)}
\end{align}
where at last we used (\ref{eq:DFhboundsgamma}).

\paragraph{Term $\bfI \bfI \bfI$.}
We have
\begin{align}
a_{BS}(v^L,w^L) - a_{BS,h}(v,w) 
&= 
(b \cdot v^L, b\cdot w^L )_\Gamma - (b \cdot v , b \cdot w )_\Gammah
\nonumber \\
&=((1 - JF_{\Gamma,h})b \cdot v^L , b\cdot w^L )_\Gamma
\end{align}
and thus we obtain the estimate
\begin{equation}
| a_{BS}(v^L,w^L) - a_{BS,h}(v,w)| 
\lesssim h^2 \| b\cdot v^L \|_{L^2(\Gamma)}\| b\cdot w^L \|_{L^2(\Gamma)}
\end{equation}
\end{proof}
\begin{lem}\label{lem:quadl} If $f_{h} = (f_{B,h},f_{S,h})$ satisfies 
the estimate
 \begin{equation}
 \|f_B - f_{B,h}^L \|_{L^2(\Omega)} + \|f_S - f_{S,h}^L \|_{L^2(\Gamma)}
 \lesssim h^2
 \end{equation}
Then it holds
\begin{equation}
|l( v^L) - l_h(v)| \lesssim h^2 \| v^L \|_{L^2(\Omega) \times L^2(\Gamma)}  
\quad \forall v \in W_h
\end{equation}
\end{lem}
\begin{proof}
We have
\begin{align}
l(v^L) - l_h (v) &= b_B(f_B,v_B^L)_\Omega - b_B(f_{B,h}, v_B)_{\Omega_h}
+ b_S(f_S,v_S^L)_\Gamma - b_S(f_{S,h}, v_S)_{\Gamma_h}
\nonumber \\
&=b_B(f_B - f_{B,h}^L JF_h, v^L_B)_\Omega + b_S(f_S - f_{S,h}^L JF_h, v^L_S)_\Gamma
\end{align}
which immediately leads to the estimate
\begin{equation}
|l(v^L) - l_h (v)| \lesssim h^2\| v^L \|_{L^2(\Omega)\times L^2(\Gamma)}
\end{equation}
\end{proof}

\subsection{Error Estimates}
\begin{thm}\label{thmenergy} The following error estimate holds
\begin{equation}\label{eq:energyest}
\Big( \tn u - u_h^L \tn^2 + \tn u - u_h \tn^2_\mcF \Big)^{1/2}\lesssim h \| u \|_{H^2(\Omega)\times H^2(\Gamma)}
\end{equation}
for small enough mesh parameter $h$.
\end{thm}
\begin{proof} Using the Strang Lemma, Lemma \ref{lem:strang}, in 
combination with the quadrature error estimates in Lemma \ref{lem:quada} 
and \ref{lem:quadl}, we obtain
\begin{align}
\Big( \tn u - u_h^L \tn^2 + \tn u - u_h \tn^2_\mcF \Big)^{1/2} 
\lesssim {} & \Big( \tn u - \pi^L_h u  \tn^2 +  \tn u - \pi_h u \tn_{\mcF}^2 \Big)^{1/2}
\nonumber \\ 
&+ \sup_{v \in W_h} \frac{a(u_h^L,v^L) - a_h(u_h,v)}{\tn v^L \tn}\nonumber
+ \sup_{v \in W_h} \frac{l( v^L) - l_h(v)}{\tn v^L \tn}
\nonumber \\
\lesssim {} &
\Big( \tn u - \pi_h^L u \tn^2 + \tn u - \pi_h u \tn^2_\mcF \Big)^{1/2} 
+ h \tn u^L_h \tn  + h^2 
\nonumber \\
\lesssim {} & h
\end{align}
Here we used the interpolation error estimates in Lemma \ref{lem:approx} 
and Lemma \ref{lem:approxface}, and the stability estimate
\begin{equation}\label{coerc_stab}
\tn u_h^L \tn \lesssim \|f\|_{L^2(\Omega)\times L^2(\Gamma)}
\end{equation} 
in the last inequality.
\end{proof}
\begin{thm} The following error estimate holds
\begin{equation}\label{eq:L2est}
\| u - u_h^L \|_{L^2(\Omega)\times L^2(\Gamma)} \lesssim  h^2 
\| u \|_{H^2(\Omega) \times H^2(\Gamma)}
\end{equation}
for small enough mesh parameter $h$.
\end{thm}
\begin{proof}
Let $\phi$ be the solution to the dual problem: find $\phi \in W$ such that 
\begin{equation}
a(v,\phi) = (v,\psi)_{L^2(\Omega) \times L^2(\Gamma)} \quad \forall v \in W
\end{equation}
where $\psi = (\psi_B,\psi_S) \in L^2(\Omega) \times 
L^2(\Gamma)$. Then we have the regularity estimate 
\begin{equation}
\| \phi \|_{H^2(\Omega) \times H^2(\Gamma)} \lesssim \|\psi\|_{L^2(\Omega)\times L^2(\Gamma)}
\end{equation}
Setting $v = u - u_h^L$, and adding and subtracting suitable terms 
we obtain
\begin{align}
(u_B - u_{B,h}^L, \psi_B )_{{\Omega}} 
+ {(u_S - u_{S,h}^L, \psi_S)_{\Gamma}} = {} & a(u - u_h^L,\phi)
\nonumber \\
= {} & a(u - u_h^L, \phi - \pi_h^L \phi ) 
+ a(u - u_h^L,\pi_h^L \phi )
\nonumber  \\
= {} & \underbrace{a(u - u_h^L,\phi - \pi_h^L \phi )}_{I}
+ \underbrace{\Big( l(\pi_h^L \phi ) - l_h(\pi_h \phi ) \Big) }_{II}
\nonumber \\ 
& + \underbrace{\Big( a_h(u_h, \pi_h \phi) 
- a(u_h^L,\pi_h^L \phi)\Big)}_{III}
+ \underbrace{ j_h(u_h ,\pi_h \phi )}_{IV}
\nonumber \\
={} & I + II + III + IV
\end{align}
\paragraph{Term $\bfI$.} Using Cauchy-Schwarz, the 
energy norm estimate (\ref{eq:energyest}), the interpolation 
estimate (\ref{eq:interpolest}) we obtain
\begin{equation}
|I|\leq \tn u - u_h^L\tn\,\tn \phi - \pi_h^L \phi \tn \lesssim h^2 { \|\psi\|_{L^2(\Omega)\times L^2(\Gamma)}}
\end{equation}
\paragraph{Term $\bfI\bfI$.} Using Lemma \ref{lem:quadl} we immediately get
\begin{equation}
|II|\lesssim h^2 { \|\psi\|_{L^2(\Omega)\times L^2(\Gamma)}}
\end{equation} 
\paragraph{Term $\bfI\bfI\bfI$.} Using Lemma \ref{lem:quada} we obtain
\begin{align}\nonumber
|a(u_h^L, \pi_h^L \phi) - a_h(u_h^L,\pi_h^L \phi)|
\lesssim {} & h^2 \| \nabla_\Gamma u_{S,h}^L \|_{L^2(\Gamma)} 
\| \nabla_\Gamma \pi_{S,h}^L \phi_{S} \|_{L^2(\Gamma)}
\\ \nonumber
& 
+ h^2 \|b\cdot u_h^L \|_{L^2(\Gamma)}\|b \cdot \pi_h^L\phi \|_{L^2(\Gamma)}
\\ 
&
+
h \| \nabla u_{B,h}^L \|_{L^2(\mcU_\delta (\Gamma)\cap\Omega)} 
\| \nabla \pi_{B,h}^L \phi_{B}  \|_{L^2(\mcU_\delta(\Gamma)\cap\Omega)} 
\end{align}
for $h \lesssim \delta \leq \delta_0$ and $h$ small enough. To show that 
the third term is actually of second order we shall use the Poincar\'e inequality 
\begin{equation}
\| v \|_{L^2(\mcU_\delta(\Gamma)\cap\Omega)} \lesssim (\delta/\delta_0)^{1/2} 
\| v \|_{H^1(\mcU_{\delta_0}(\Gamma)\cap\Omega)}, \quad 0<\delta \leq \delta_0
\end{equation}
See \cite{EllRan12} for a proof of this inequality. We proceed in the 
following way
\begin{align}
\| \nabla \pi_{B,h}^L \phi_{B}  \|_{L^2(\mcU_\delta(\Gamma)\cap\Omega)} 
&\leq
\| \nabla (\pi_{B,h}^L \phi_{B} - \phi_B) \|_{L^2(\mcU_{\delta}(\Gamma)\cap\Omega)} 
+
\| \nabla \phi_B \|_{L^2(\mcU_{\delta}(\Gamma))}
\nonumber \\
&\lesssim (h + \delta^{1/2}) \| \phi_B \|_{H^2(\mcU_{\delta_0}(\Gamma)\cap\Omega)}
\nonumber \\
&\lesssim (h + h^{1/2}) \| \phi_B \|_{H^2({\Omega})} \label{phibound}
\end{align}
where we used the fact that $\delta$ can actually be chosen such that 
$\delta \sim h$, see Lemma \ref{lem:quada}, and the interpolation error estimate 
(\ref{eq:interpolest}). The term $\| \nabla u_{B,h}^L \|_{L^2(\mcU_\delta(\Gamma))}$ can be estimated 
using the same technique but we employ the energy norm error estimate 
(\ref{eq:energyest}) instead
\begin{align}\nonumber
\| \nabla u_{B,h}^L \|_{L^2(\mcU_\delta(\Gamma)\cap\Omega)} 
&\lesssim \| \nabla (u_{B,h}^L - u_B)\|_{L^2(\mcU_\delta(\Gamma)\cap\Omega)} 
+ \| \nabla u_B\|_{L^2(\mcU_\delta(\Gamma)\cap \Omega)}
\\ \nonumber
&\lesssim \tn u - u_h \tn 
+ \delta^{1/2}\| \nabla u_B\|_{L^2(\mcU_\delta(\Gamma)\cap \Omega)}
\\ \label{ubound}
&\lesssim (h + h^{1/2} \| u \|_{H^2(\Omega)\times H^2(\Gamma)}
\end{align}
Combining (\ref{phibound}) and (\ref{ubound}) we obtain
\begin{equation}
III \lesssim (h^2 +  h(h + h^{1/2})^2) \|\psi\|_{L^2(\Omega)\times L^2(\Gamma)}\lesssim h^2 \|\psi\|_{L^2(\Omega)\times L^2(\Gamma)}
\end{equation}

\paragraph{Term $\bfI\bfV$.} Using the fact that the jump term is 
consistent we obtain
\begin{equation}
|IV|= |j_h(u - u_h,\phi - \pi_h \phi)|\leq \tn u - u_h\tn_\mcF 
\tn \phi - \pi_h \phi \tn_\mcF \lesssim h^2 {\|\psi\|_{L^2(\Omega)\times L^2(\Gamma)}}
\end{equation}
where we used the energy estimate in Theorem \ref{thmenergy} and the 
interpolation estimate in Lemma \ref{lem:approx}. 

We conclude the proof by collecting the estimates of Terms $I-IV$ 
and taking the supremum over all
$\psi$ such that $\|\psi\|_{L^2(\Omega)\times L^2(\Gamma)} = 1$.
\end{proof}

\section{Estimate of the Condition Number}

Due to the different dimensions of the two coupled differential equations 
at the surface we shall see that it is natural to precondition the system 
in such a way that we seek $(v_{B,h},v_{S,h})$ such that the solution $(u_{B,h},u_{S,h})$ of (\ref{eq:fem}) is given by 
\begin{equation}
(u_{B,h}, u_{S,h}) = (v_{B,h},h^{1/2} v_{S,h})
\end{equation}
The corresponding variational problem for $v_h=(v_{B,h},v_{S,h})$ takes 
the form: find $v = (v_B,v_S) \in W_h$ such that 
\begin{equation}
\tilde{A}_h(v,w) = \tilde{L}_h(w) \quad \forall w \in W_h 
\end{equation}
where the bilinear forms are defined by
\begin{equation}\label{eq:Atilde}
\tilde{A}_h(v,w) = A_h((v_B,h^{1/2} v_S), (w_B, h^{1/2} w_S)), 
\quad \tilde{L}_h(w) = L_h((w_B,h^{1/2} w_S))
\end{equation}
We shall now estimate the condition number of the stiffness matrix 
$\tilde{A}$ associated with the bilinear form $\tilde{A}_h(\cdot,\cdot)$. 
Let $\{\varphi_{B,i}\}_{i=1}^{N_B}$ and $\{\varphi_{S,i} \}_{i=1}^{N_S}$ 
be the standard piecewise linear basis functions in $V_{B,h}$ and
$V_{S,h} \oplus \langle 1_{\Gamma_h}\rangle$, respectively. Note that we 
have added the one dimensional space $\langle 1_{\Gamma_h}\rangle$ of constant functions on $\Gamma_h$. 
Define the following basis in the product space 
$V_{B,h} \times V_{S,h}\oplus \langle 1_{S,h}\rangle$: 
\begin{equation}
\varphi_i =
\begin{cases}
(\varphi_{B,i},0) & 1\leq i \leq N_B
\\
(0,\varphi_{S,i-N_B}) & 1 + N_B \leq i \leq N = N_B + N_S 
\end{cases}
\end{equation}
The expansion $v = \sum_{i=1}^N \widehat{v}_i \varphi_i$ defines an 
isomorphism 
\begin{align}
V_{B,h} \times V_{S,h}/\langle 1_{\Gamma_h} \rangle 
\oplus \langle 1_{S,h}\rangle  
&\rightarrow 
\bbR^{N_B} \times \bbR^{N_S} / \langle 1_{\bbR^{N_S}}\rangle \oplus 
\langle 1_{\bbR^{N_S}}\rangle
\\
(v_B, v_S \oplus \overline{v}_S 1_{S,h}) 
&\mapsto 
(\widehat{v}_B,\widehat{v}_S \oplus \overline{v}_S 1_{\bbR^{N_S}})  
\end{align}
where $v_S$ is the unique element in the equivalence classes of 
$V_{S,h}/\langle 1_{\Gamma_h} \rangle$ with 
$\int_{\Gamma_h} v_S =0$ and $\overline{v}_S=|\Gamma_h|^{-1} \int_{\Gamma_h} v_S$
is the meanvalue of $v_S$. 
If we introduce the mesh dependent $L^2$-norm
\begin{equation}
\| v \|_h^2 = \| v_B \|^2_{L^2(\mcN_{B,h})} + \| v_S \|^2_{L^2(\mcN_{S,h})} 
\end{equation}
where the sets $\mcN_{B,h}$ and $\mcN_{S,h}$ are defined in (\ref{eq:mcnmcs}),
we have the following standard estimate 
\begin{equation}\label{rneqv}
c h^{-d} \| v \|^2_h \lesssim | \widehat{v} |^2_N \lesssim C h^{-d} \| v \|^2_h 
\end{equation}

Let $\tilde{A}$ be the stiffness matrix with elements $a_{ij} 
= \tilde{A}_h(\varphi_i,\varphi_j) + J_h(\varphi_i,\varphi_j)$. The stiffness 
matrix is symmetric and has a one dimensional kernel consisting of a 
constant functions $v =(v_B,v_S)$, that satisfy $b\cdot v = b_B v_B - b_S v_S=0$. 
We shall estimate the condition number of $\tilde{A}$ as an operator on the 
invariant space $V=\bbR^{N_B} \times \bbR^{N_S}/\langle 1_{\bbR^{N_S}}\rangle$ 
defined by
\begin{equation}
\kappa(\tilde{A}) = | \tilde{A} |_V |\tilde{A}^{-1} |_V
\end{equation}
where $|x|^2_N = \sum_{i=1}^N x_i^2$ for $x \in \bbR^N$ and 
$|\tilde{A}|_V = \sup_{X\in V \setminus \{ 0\}} \frac{|\tilde{A} x|_N}{|x|_N}$ 
for $\tilde{A} \in \bbR^{N\times N}$. Next we introduce the discrete 
energy norm 
\begin{equation}
\tn v \tn_h^2 = A_h(v,v) = a_h(v,v) + j_h(v,v)
\end{equation}
The proof of the estimate of the condition number follow the 
approach presented in \cite{ErnGer06} and rely on a Poincar\'e  
and an inverse inequality which we prove next. 

\begin{lem} (Poincar\'e inequality) {Independently of the mesh/boundary
  intersection} it holds that
\begin{equation}\label{eq:poincarecond}
\| (v_B, v_S) \|_h \lesssim \tn ( v_B,h^{1/2} v_S) \tn_h
\quad \forall (v_B, v_S) \in W_h
\end{equation}
\end{lem} 
\begin{proof} 
Using Lemma 3.3 in \cite{BurHanLar13} and then adding and subtracting 
suitable terms and using the triangle inequality followed by a Poincar\'e 
inequality we obtain  
\begin{align}
\|v_S\|^2_{L^2(\mcN_{S,h})} 
&\lesssim h \| {v_S} \|^2_{L^2(\Gamma_h)} + h j_S({v_S},{v_S}) \nonumber
\\
&\lesssim h \| \nabla_{\Gamma_h} v_S \|^2_{L^2(\Gamma_h)} + j_S({h^{1/2}v_S},{h^{1/2}v_S})\nonumber
\\ \label{eq:poincareb}
&\lesssim \tn (v_B , h^{1/2} {v_S}) \tn_h^2
\end{align}
Note that the Poincar\'e inequality is applicable on $\Gamma_h$ since $\int_{\Gamma_h} v_S=0$.

Next using the control provided by the jump term 
$J_B(\cdot,\cdot)$ followed by a Poincar\'e inequality we obtain
\begin{align} \nonumber
\|v_B\|^2_{L^2(\mcN_{B,h})} 
\lesssim {} &\| v_B \|^2_{L^2(\Omega_h)} + h^3 j_B(v_B, v_B)
\\ \nonumber
\lesssim  {} & \| P_0 v_B \|^2_{L^2(\Omega_h)} + \| \nabla v_B \|^2_{L^2(\Omega_h)} 
+ h^3 J_B(v_B,v_B)
\\ \nonumber
\lesssim {} & \| P_0 v_B \|^2_{L^2(\Gamma_h)} + \tn (v_B, h^{1/2} v_S) \tn_h^2
\\ \nonumber
\lesssim {} & \| (I - P_0) v_B \|^2_{L^2(\Gamma_h)} + \| v_B \|^2_{L^2(\Gamma_h)} 
+ \tn (v_B, h^{1/2} v_S) \tn_h^2
\\ \nonumber
\lesssim {} & \| \nabla v_B \|^2_{L^2(\Omega_h)} 
+ b_B^{-2} \| b_B v_B - b_S h^{1/2} v_S\|_{L^2(\Gamma_h)}^2 
\\ \nonumber
& + b_B^{-2} b_S^2  \|h^{1/2} v_S\|^2_{L^2(\Gamma_h)} 
+ \tn (v_B, h^{1/2} v_S) \tn_h^2 
\\
  \label{eq:poincarea}
\lesssim {} &b_B^{-2} \| h^{1/2} v_S \|^2_{L^2(\Gamma_h)} 
+ \tn (v_B, h^{1/2} v_S) \tn_h^2
\end{align}
Here $P_0 v_B$ is the $L^2$-projection of $v_B$ onto constant functions on 
$\Omega_h$ and we added and subtracted suitable functions to control $P_0 v_B$ 
using the coupling term together with 
the control of $\|h^{1/2} v_S\|^2_{\Gamma_h}$ provided by (\ref{eq:poincareb}) 
and the fact that the constant $b_B>0$. Furthermore, the first inequality {in (\ref{eq:poincarea})}
is a consequence of the inverse inequality 
\begin{equation}\label{eq:pairjump}
\| v \|^2_{L^2(K_1)} \lesssim \| v \|^2_{L^2(K_2)} 
+ h^3 \|[n_F \cdot \nabla v ]\|^2_{L^2(F)}
\quad \forall v \in V_{B,h}
\end{equation}
that holds for each pair of elements $K_1$ and $K_2$ that share a face $F$. 
Iterating the inequality (\ref{eq:pairjump}) we may control the elements at 
the boundary in terms of the elements in the interior of $\Omega_h$ as follows
\begin{equation}\label{eq:chainjump}
\| v \|^2_{L^2(K_1)} \lesssim \| v \|^2_{L^2(K_N)} + \sum_{i=1}^{N-1} h^3 \|[n_F \cdot \nabla v ]\|^2_{L^2(F_i)}\quad \forall v \in V_{B,h}
\end{equation}
see \cite{MasLar14} for further details. Note that for sufficiently small 
mesh size the length $N$ of the shortest chain of elements that share an edge 
between an element that intersects the boundary and an interior element 
is uniformly bounded.

Combining the two estimates (\ref{eq:poincareb}) and (\ref{eq:poincarea}) 
the lemma follows directly.
\end{proof}
\begin{lem} (Inverse inequality)  {Independently of the mesh/boundary
  intersection} it holds that
\begin{equation}\label{inverse}
\tn (v_B,h^{1/2}{v_S}) \tn_h^2 \lesssim h^{-2} \| (v_B,{v_S}) \|^2_{h} 
\quad \forall (v_B,v_S) \in W_h
\end{equation}
\end{lem}
\begin{proof} Using standard estimates we obtain the following three estimates
\begin{align}
{b_B} k_B \| \nabla v_B \|_{L^2(\Omega_h)}^2 + \tau_B h^3 j_B(v_B,v_B) \lesssim 
h^{-2} \| v_B \|^2_{L^2(\mcN_{B,h})} \lesssim h^{-2} \| (v_B,v_S) \|_h^2
\end{align}
\begin{align}
\|{b_B} v_B - {b_S} h^{1/2} v_S \|^2_{L^2(\Gamma_h)} 
\lesssim h^{-1} {b_B} \|v_B \|^2_{L^2(\mcN_{S,h})} + {b_S} \|{v_S} \|^2_{L^2(\mcN_{S,h})} \lesssim h^{-2} \| (v_B,{v}_S) \|_h^2
\end{align}
\begin{align}\nonumber
{b_S} k_S h \| \nabla_{\Gamma_h} {v_S} \|_{L^2(\Gamma_h)}^2 
+ \tau_S h j_S({v_S},{v_S}) ) 
&\lesssim (b_S k_S + \tau_S ) \| \nabla {v_S} \|_{L^2(\mcN_{S,h})}^2
\lesssim h^{-2} \| {v_S} \|^2_{L^2(\mcN_{S,h})}
\end{align}
and thus the proof is complete.
\end{proof}

Finally, we are ready to prove our final estimate of the condition number.

\begin{thm} The following estimate of the condition number of the stiffness 
matrix holds {independently of the mesh/boundary
  intersection}
\begin{equation}
\kappa( \tilde{A} )\lesssim h^{-2}
\end{equation}
\end{thm}
\begin{proof} We need to estimate $| \tilde{A} |_V$ and $|\tilde{A}^{-1}|_V$. Starting with 
$| \tilde{A} |_V$ we have
\begin{align}\nonumber
|\tilde{A} \widehat{v}|_V &= \sup_{\widehat{w} \in \bbR^N } \frac{(\widehat{w},\tilde{A}\widehat{v})_N}{| \widehat{w} |_N}
\\ \nonumber
&= \sup_{w \in W_h}  
\frac{A_h((v_B,h^{1/2} v_S),(w_B,h^{1/2} w_S))}{\tn (w_B,h^{1/2}w_S) \tn_h} \frac{\tn (w_B,h^{1/2} w_S) \tn_h}{\| (w_B,w_S) \|_h }\frac{\| (w_B,w_S) \|_h}{| \widehat{w} |_N}
\\ \nonumber
&\lesssim h^{(d-2)/2} \tn(v_B,h^{1/2} v_S)\tn_h 
\\
&\lesssim h^{d-2}| \widehat{v} |_N
\end{align}
where at last we used the estimate
\begin{equation}
\tn (v_B, h^{1/2} v_S ) \tn_h \lesssim h^{-1} \|(v_B,v_S) \|_{h} \lesssim h^{(d-2)/2} |\widehat{v}|_N
\end{equation}
together with (\ref{inverse}) and (\ref{rneqv}). Thus
\begin{equation}\label{Aest}
|\tilde{A}|_V \lesssim h^{d-2}
\end{equation}

Next we turn to the estimate of $|\tilde{A}^{-1}|_V$.
Using (\ref{rneqv}) and (\ref{eq:poincarecond}), we get
\begin{align}
|\widehat{v}|^2_N 
\lesssim {} & h^{-d} \tn (v_B,h^{1/2} v_S) \tn_h^2
\lesssim h^{-d} A_h((v_B,h^{1/2}v_S),(v_B,h^{1/2}v_S)) \nonumber
\\
\lesssim {} & h^{-d} (\widehat{v}, \tilde{A} \widehat{v})_N 
\lesssim h^{-d} |\widehat{v}|_N |\tilde{A}\widehat{v}|_N
\end{align}
and thus we conclude that $|\widehat{v}|_N \leq C h^{-d}|\tilde{A}\widehat{v}|_N$. Setting $ \widehat{v} = \tilde{A}^{-1} \widehat{w}$ we obtain
\begin{equation}\label{Ainvest}
|\tilde{A}^{-1}|_N \lesssim h^{-d}
\end{equation}
Combining estimates (\ref{Aest}) and (\ref{Ainvest}) of $|\tilde{A}|_N$ and 
$|\tilde{A}^{-1}|_N$ the theorem follows.
\end{proof}

\section{Numerical results}
We consider an example where the domain $\Omega$ is the unit sphere, $k_B=k_S=1$, $b_B=b_S=1$, and $f_B$ and $f_S$ are choosen such that the exact solution is as in~\cite{EllRan12} given by
\begin{align}\label{eq:exactsol}
u_B&=e^{(-x(x-1)-y(y-1))} \nonumber \\
u_S&=(1+x(1-2x)+y(1-2y))e^{(-x(x-1)-y(y-1))}
\end{align}
We study the convergence rate of the numerical solution $u_h=(u_{B,h},u_{S,h})$ and the condition number of the system matrix using the proposed finite element method. A direct solver is used to solve the linear systems. The stabilization parameters $\tau_B=\tau_S=10^{-2}$. We use a structured mesh for $\Omega_0$ and the mesh parameter $h=h_x=h_y=h_z$. 

To represent the boundary $\Gamma$ we use the standard level set method. We define a piecewise linear approximation to the distance function on $\mcK_{0,h}$ and $\Gamma$ is approximated as the zero level set of this approximate distance function. Thus, $\Gamma_h$ is represented by linear segments on $\mcK_{0,h}$. The normal vectors are computed from the linear segments. 

The solution $u_{S,h}$ with $h=0.13125$ and the triangulation of $\Gamma_h$ are shown in Fig.~\ref{fig:sol}. The convergence of $u_h$ in both the $L^2$ norm and the $H^1$ norm are shown in Fig.~\ref{fig:conv}. We have as expected first order convergence in the 
$H^1$ norm and second order convergence in the $L^2$ norm. The spectral condition number of the matrix $\tilde{A}$ associated with the bilinear form $\tilde{A}_h(\cdot,\cdot)$ (see equation~\eqref{eq:Atilde}) is shown for different mesh sizes in Fig.~\ref{fig:cond}.
\begin{figure}
\begin{center} 
\includegraphics[width=0.6\textwidth]{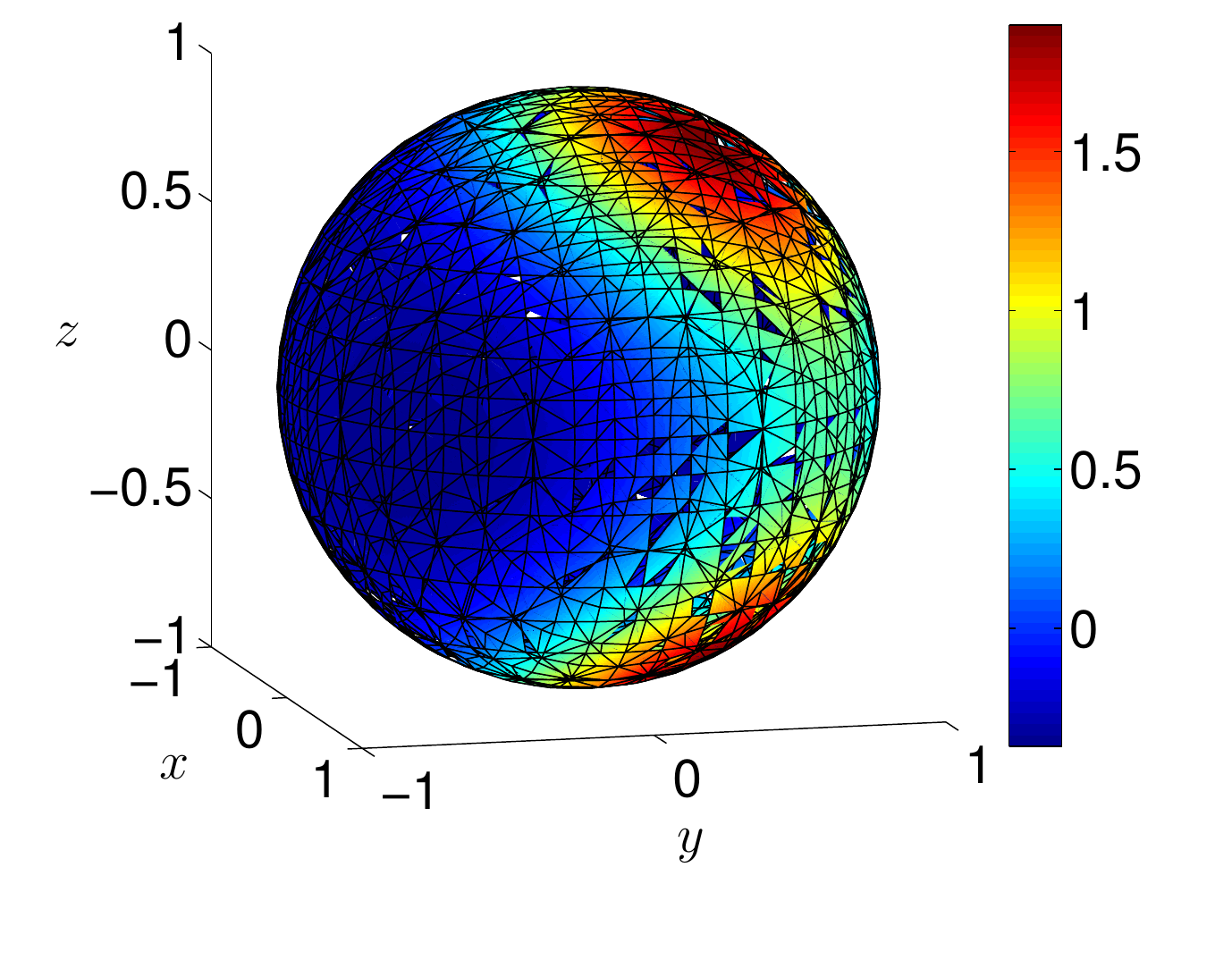}
\caption{The solution $u_{S,h}$ with $h=0.13125$. \label{fig:sol}}
\end{center}
\end{figure}

\begin{figure}
\begin{center} 
\includegraphics[width=0.5\textwidth]{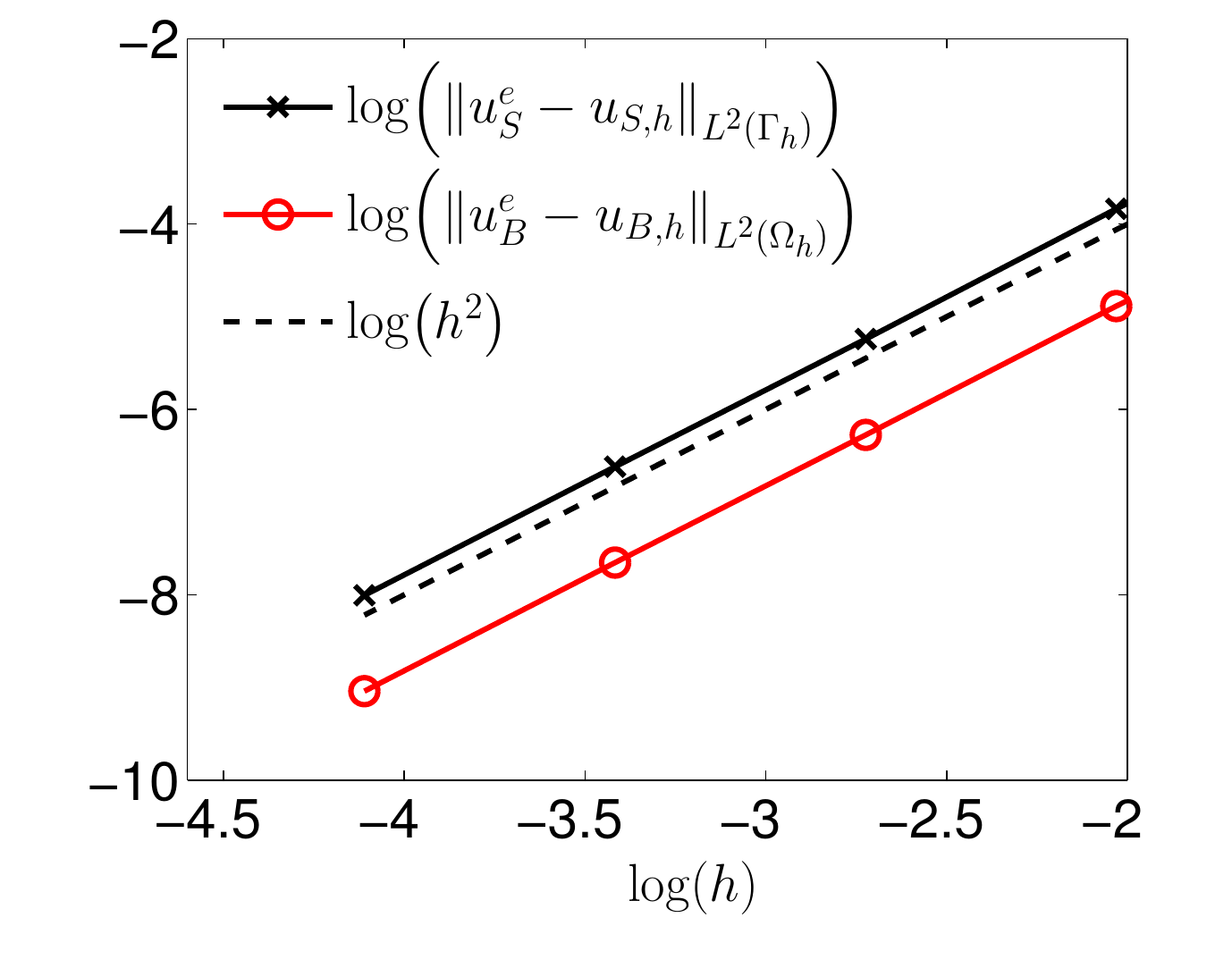}
\includegraphics[width=0.5\textwidth]{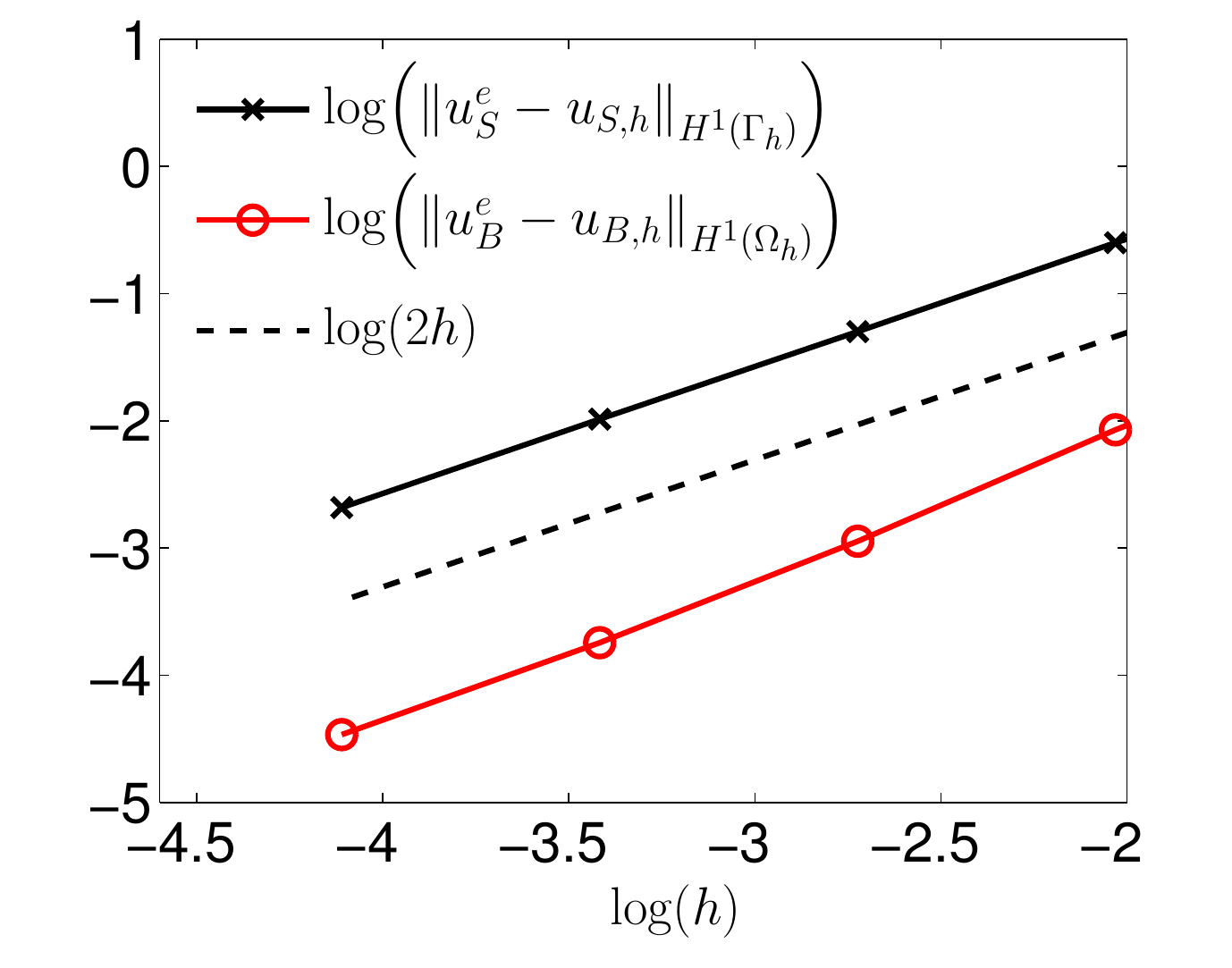}
\caption{Convergence of $u_B$ and $u_S$. Upper panel: The error measured in the $L^2$ norm versus mesh size. Lower panel: The error measured in the $H^1$ norm versus mesh size.\label{fig:conv}}
\end{center}
\end{figure}

\begin{figure}
\begin{center} 
\includegraphics[width=0.5\textwidth]{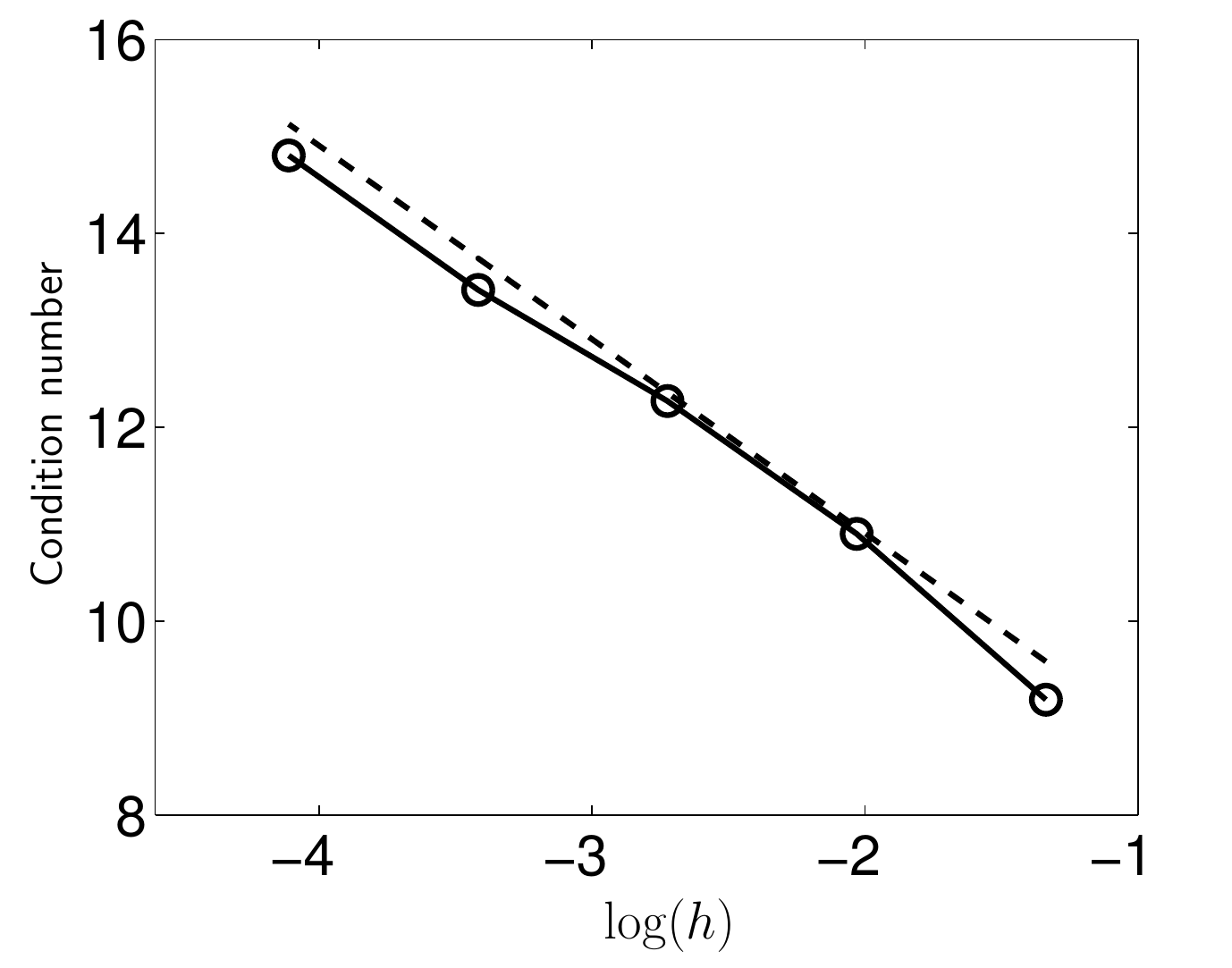}
\caption{The spectral condition number of the matrix $\tilde{A}$ versus mesh size. The dashed line is proportional to $h^{-2}$. \label{fig:cond}}
\end{center}
\end{figure}

\section*{Appendix}
Here we will give some details on the inequalities
\eqref{eq:gammahest}. First we recall that 
\begin{equation}\label{app:qh}
q_h(x) = x + \gamma_h(x)n(x) \quad x \in \Gamma
\end{equation}
Now using the defintion of the closest point mapping
\begin{equation}
y = p(y) +  \rho(y)n^e(y) \quad y \in \Gammah
\end{equation}
Setting $x = p(y)$ in (\ref{app:qh}) we have
\begin{equation}
y = p(y) + \gamma_h(p(y)) n^e(y)\quad y \in \Gammah
\end{equation}
and therefore, by uniqueness, 
$\rho(y) = \gamma_h(p(y)), \forall y \in \Gammah$. Thus we 
have $\gamma_h = \rho^L$ and we immediately obtain the first 
inequality in (\ref{eq:gammahest}) since
\begin{equation}
\| \gamma_h\|_{L^\infty(\Gamma)} = \|\rho^L\|_{L^\infty(\Gamma)}
= \|\rho\|_{L^\infty(\Gammah)}\lesssim h^2
\end{equation}
Next using (\ref{eq:LiftDerGamma}) we have the identity
\begin{equation}
\nabla_\Gamma \gamma_h = \nabla_\Gamma \rho^L 
= DF_{h,\Gamma}^T (\nabla_\Gammah \rho)^L 
= DF_{h,\Gamma}^T (P_\Gammah n^e)^L
\end{equation}  
Estimating the right hand side using (\ref{eq:DFhboundsgamma}) and 
(\ref{est:PGammahn}) we finally obtain  
\begin{equation}
\|\nabla_\Gamma \gamma_h \|_\Gamma 
\lesssim \|\nabla_\Gammah \rho \|_\Gammah 
\lesssim \| P_\Gammah n^e \|_\Gammah 
\lesssim h
\end{equation}
which is the second bound in (\ref{eq:gammahest}).

%

\bibliographystyle{elsarticle-num}

\end{document}